\documentclass[11pt,reqno]{amsart} 

\usepackage[letterpaper,left=1.5in,right=1.5in]{geometry}

\usepackage{amssymb}
\usepackage[mathscr]{eucal}

\DeclareMathAlphabet{\mathpzc}{OT1}{pzc}{m}{it}

\usepackage{mathtools} 
\usepackage{yhmath} 
\usepackage{shuffle}
\usepackage{extarrows}

\usepackage{color}
\usepackage[width=.75\textwidth]{caption}
\usepackage{subcaption}

\usepackage{tikz}
\usetikzlibrary{cd,patterns}

\usepackage{graphicx,multicol}

\usepackage[%
colorlinks,
linkcolor=blue!80!black, 
citecolor=green!80!black, 
urlcolor=magenta!80!black,
]{hyperref}

\usepackage[backend=biber,maxbibnames=99,sorting=nyt]{biblatex}
\newtheorem{theorem}{Theorem}
\newtheorem{corollary}{Corollary}
\newtheorem{lemma}{Lemma}
\newtheorem{proposition}{Proposition}

\theoremstyle{definition}
\newtheorem{definition}{Definition}
\newtheorem{remark}{Remark}
\newtheorem{example}{Example}

\definecolor{myred}{rgb}{.7,.1,.1}
\definecolor{myblue}{rgb}{.1,.1,.7}
\definecolor{mygreen}{rgb}{.1,.6,.1}
\definecolor{mygray}{rgb}{0.5,0.5,0.5}
\definecolor{mymauve}{rgb}{0.58,0,0.82}
\definecolor{lablue}{rgb}{0,0.9,0.93}

\usepackage[textwidth=1in,textsize=normalsize]{todonotes} 

\newcommand{\demph}[1]{\textcolor{myblue}{\emph{#1}}}

\def\qand{\quad\hbox{and}\quad}

\def\ds{\displaystyle}

\mathchardef\mhyphen="2D 

\def\End{\operatorname{\mathrm{End}}}
\def\Span{\operatorname{\mathrm{span}}}
\def\id{\mathrm{id}}
\def\tw{\mathrm{tw}}

\def\into{\hookrightarrow}
\def\onto{\twoheadrightarrow}

\def\port{\mathsf{port}}

\def\splus{\mathbin{\boldsymbol{+}}}
\def\stimes{\mathbin{\boldsymbol{\times}}} 
\def\sdot{\mathbin{\boldsymbol{\cdot}}} 
	\def\tsdot{\textstyle{\sdot}}
\def\scirc{\mathbin{\boldsymbol{\circ}}} 

\def\dototimes{\mathrel{\dot{\otimes}}}

\def\Spec{\mathsf{Sp}}
\def\Set{\mathsf{Set}}
\def\Vec{\mathsf{Vec}}

\def\bimon{\mathsf{Bimon}}

\def\sym{\mathrm{Sym}}
\def\qsym{\mathrm{QSym}}
\def\nsym{\mathrm{NSym}}
\def\ncsym{\mathrm{NCSym}}
\def\ncqsym{\mathrm{NCQSym}}
\newcommand{\rqsym}[1][r]{{}^{#1}\qsym}

\def\K{\mathcal K}
\def\Kbar{\overline{\mathcal K}}
\def\Kvee{{\mathcal K}^\vee}
\def\Kbv{\overline{\mathcal K}^\vee}

\def\bfPi{\mathbf{\Pi}}
\def\bfG{\mathbf G}
\def\bfE{\mathbf E}
\def\bfL{\mathbf L}
\def\bij{\mathbf{bij}}
\def\cyc{\mathbf{cyc}}

\def\bfB{\mathbf{B}}

\def\b{\mathbf b} 
\def\d{\mathbf d} 
\def\h{\mathbf h}
\def\p{\mathbf p}
\def\q{\mathbf q}
\def\a{\mathbf a}
\def\bfc{\mathbf c}

\def\rfp{\mathrm p}
\def\rfG{\mathrm G}
\def\rfE{\mathrm E}
\def\rfL{\mathrm L}

\def\rfB{\mathrm{B}}

\def\ZZ{\mathbb Z}

\newcommand{\teebar}[2][r]{\rspec[#1]{#2}}
\newcommand{\tee}[2]{\operatorname{\cfC_{#1}}(#2)}

\newcommand{\rspec}[2][r]{\prescript{\,#1\,}{}{#2}}

\newcommand{\upspec}[2][r]{{#2}_{\underbracket[.6pt][.25ex]{\scriptstyle\hspace{.15em}#1\hspace{.05em}}}}
\newcommand{\downspec}[2][r]{{#2}_{\wideparen{\scriptstyle\hspace{.1em}#1\hspace{.1em}}}}

\newcommand{\rc}[1][r]{\prescript{#1}{}{\bfc}}
\def\sc{\rc[s]}
\def\cfC{\mathcal{C}}
\def\cfT{\mathcal{T}}
\def\cfS{\mathcal{S}}

\newcommand{\threegraph}[6]{%
\vcenter{\hbox{%
\begin{tikzpicture}[scale=0.75, shorten <=4pt, shorten >=4pt] 
  \coordinate (A) at (0, 0);
  \coordinate (B) at (0.55, -0.85);  
  \coordinate (C) at (1.1, 0);

  \node at (A) {$#1$};
  \node at (B) {$#2$};
  \node at (C) {$#3$};

  \ifnum#4=1 \draw (A) -- (B); \fi
  \ifnum#5=1 \draw (A)--(C); \fi
  \ifnum#6=1 \draw (B)--(C); \fi
\end{tikzpicture}%
}}%
}

\newcommand{\threegraphline}[7][1]{%
\vcenter{\hbox{%
\begin{tikzpicture}[scale=0.75, shorten <=3pt, shorten >=3pt] 
  \coordinate (A) at (0, 0);
  \coordinate (B) at (0.6, -0.2);
  \coordinate (C) at (1.2, 0);

  \node at (A) {$#2$};
  \node at (B) {$#3$};
  \node at (C) {$#4$};
\ifnum#1=0
  \ifnum#5=1 \draw[bend right=25] (A) to (B); \fi
  \ifnum#6=1 \draw[bend left=55] (A) to (C); \fi
  \ifnum#7=1 \draw[bend right=25] (B) to (C); \fi
\fi
\ifnum#1=1
  \ifnum#5=1 \draw[] (A) to (B); \fi
  \ifnum#6=1 \draw[] (A) to (C); \fi
  \ifnum#7=1 \draw[] (B) to (C); \fi
\fi
\end{tikzpicture}%
}}%
}

\newcommand{\twograph}[3]{%
\vcenter{\hbox{%
\begin{tikzpicture}[scale=0.8, shorten <=5.5pt, shorten >=9pt] 
  \coordinate (A) at (0,0);
  \coordinate (B) at (1,0);

  \node at (A) {$#1$};
  \node at (B) {$#2$};

  \ifnum#3=1 \draw (A)--(B); \fi
\end{tikzpicture}%
}}%
}

\newcommand{\onegraph}[1]{%
\vcenter{\hbox{%
\begin{tikzpicture}[scale=0.9, shorten <=4pt, shorten >=4pt] 
  \node at (0,0) {$#1$};
\end{tikzpicture}%
}}%
}

\begin{document}

\title{Hopf Substitutions in Species}

\author{Aaron Lauve}
\address{Department of Mathematics and Statistics, Loyola University Chicago -- Chicago, IL, USA}
\email{alauve@luc.edu} 


\author{Anthony Lazzeroni}
\address{Department of Mathematics, North Park University -- Chicago, IL, USA}
\email{aalazzeroni@northpark.edu}

\date{April 12, 2026}

\begin{abstract}%
In the theory of species, the species $\mathbf{L}$ of linear orders and the substitution operation $\boldsymbol{\circ}$ combine for a compelling result: given any positive comonoid $\mathbf{p}$, $\mathbf{L}\boldsymbol{\circ}\mathbf{p}$ carries the structure of Hopf monoid, freely generated by $\mathbf{p}$. Leaving aside the universal property this implies, we ask, \emph{for which $\mathbf{b}$ does $\mathbf{b}\boldsymbol{\circ}\mathbf{p}$ carry the structure of Hopf monoid?} After answering this question, we look at basic properties of our construction. We also extend a result of the present authors, on interpolation in species, to this new context.
\end{abstract}

\keywords{Hopf monoids, species, graded Hopf algebras, operads}

\maketitle



\section{Introduction} 
We begin in the monoidal category of (vector) species under Cauchy product $(\Spec,\sdot)$. After the work of Joni--Rota, Joyal, and others \cite{joni1979coalgebras,joyal1981species,schmitt1993hopf,stover1993equivalence,bergeron1998combinatorial}, this seems to be the natural setting for understanding Hopf algebraic combinatorics. (Starting from here, the ``Fock functors'' introduced by Stover take Hopf monoids in $(\Spec,\sdot)$ to a wealth of Hopf algebras in $(\Vec, \otimes)$ of combinatorial interest.) In \cite{aguiar2010monoidal}, a systematic study of different operations ($\splus,\sdot,\stimes,\scirc$) and the Fock functors ($\K, \Kbar, \Kvee, \Kbv$) on species is undertaken to better organize the spoils of the work cited above. For example, one may ask, \emph{if $\b$ and $\p$ are Hopf monoids in species, what about $\b\sdot \p$ or $\b\stimes\p$?} The answer here is, ``yes,'' see \cite[\S\S~1.2.7 \& 8.13.1]{aguiar2010monoidal} for proofs. One benefit of this perspective is immediate: even exceptionally rich objects in $(\Vec,\otimes)$ originate from simple building blocks with these operations. (A salient example here would be the infamous Malvenuto--Reutenauer Hopf algebra of permutations, arising from $\bfL\stimes\bfL^*$. Here $
\bfL$ is the Hopf monoid of linear orders; see \cite[\S 17.2.1]{aguiar2010monoidal} for details.)

In this note, we ask the same question about $\b\scirc \p$. Our answer is, briefly, ``whenever $\b$ is a cocommutative linearized Hopf monoid (and $\p$ is any positive comonoid).'' This extends the constructions of $\cfT(\p)=\bfL\scirc\p$ and $\cfS(\p)=\bfE\scirc\p$ in \cite{aguiar2010monoidal} to a much larger class of Hopf monoids. 

After introducing the new object, $\tee{\b}{\p}$, we give basic structural results and related Hopf algebraic examples. Our final result is an extension of the main theorem of \cite{lauve202xinterpolation} to the present context: we give a one-parameter family of Hopf monoids $\bigl\{\rspec{\tee{\b,\d}{\p,\q}}\bigr\}_{r\geq1}$ interpolating between $\tee{\b}{\p}$ and $\tee{\d}{\q}$ under suitable conditions on $\b,\d,\p,\q$.


\section{Preliminaries}
\label{sec:prelims}

In this section we briefly catalog what is needed, referring the reader to \cite{aguiar2010monoidal, bergeron1998combinatorial} for omitted definitions, details and applications of the theory of species.

\subsection{Species and basic constructions}
\label{sec:species}

Let $\Set$ (respectively, $\Vec$) be the category of sets (vector spaces), with morphisms being arbitrary set maps (vector space maps). Let $\Set^{\times}$ be the category of finite sets with bijections as morphisms. A \demph{set species} is a functor $\rfp:\Set^{\times} \to \Set$ (assigning a set $\rfp[I]$ for each $|I|<\infty$, and bijections $\rfp[\sigma]:\rfp[I] \to \rfp[J]$ for each bijection $\sigma:I\to J$). A \demph{(vector) species} $\p:\Set^{\times} \to \Vec$ is defined similarly.

Call $\p$ \demph{linearized} if there is a set species $\rfp$ so that $\p[I] = \Bbbk\rfp[I]$ for all $I$ and $\p[\sigma] = \Bbbk\rfp[\sigma]$ for all $\sigma:I\to J$. 

\begin{example}\label{ex:common examples}
Four common linearized species that will feature prominently in what follows are: 
the \emph{trivial species} $\mathbf 1$; 
the \emph{exponential species} $\mathbf E$; 
the species $\bfL$ of \emph{linear orders} ({\it i.e.}, total orderings or lists) of $I$; and 
the species $\bfG$ of \emph{finite simple graphs} with vertex set $I$. Definitions as in \cite{aguiar2010monoidal}. 
For instance, $\bfG[\{a,b,c\}]$ is the $\Bbbk$-span of 
\begin{gather}
     \label{eq:graphs-abc}
\threegraph{a}{b}{c}{1}{1}{1}
\ \ \ 
\threegraph{a}{b}{c}{1}{1}{0}
\ \ \ 
\threegraph{a}{b}{c}{1}{0}{1}
\ \ \ 
\threegraph{a}{b}{c}{1}{0}{0}
\ \ \ 
\threegraph{a}{b}{c}{0}{1}{1}
\ \ \ 
\threegraph{a}{b}{c}{0}{1}{0}
\ \ \ 
\threegraph{a}{b}{c}{0}{0}{1}
\ \ \ 
\threegraph{a}{b}{c}{0}{0}{0}
\end{gather}
(The first and last graph are the complete graph and edgeless graph, respectively.) 
\end{example}

Call $\p$ \demph{connected} if $\p[\emptyset]\cong\Bbbk$ and \demph{positive} if $\p[\emptyset]=\{0\}$.  
%
%
%
For $n\geq 0$, let $\p_n$ be the new species concentrated in cardinality $n$; define $\p_+$ analogously.
\[
\p_n[I] = 
\begin{cases}
    \p[I] & \text{if } |I|=n, \\
    \{0\} & \text{otherwise};
\end{cases}
\qquad
\p_+[I] = 
\begin{cases}
    \p[I] & \text{if } |I|>0, \\
    \{0\} & \text{otherwise}.
\end{cases}
\]

Below and throughout: $S \sqcup T=I$ indicates an ordered decomposition of a set $I$ into (possibly empty) subsets; while $X\vdash I$ indicates an ordinary, unordered partition of $I$.

\begin{definition}\label{def:operations}
The \demph{sum}, \demph{Hadamard product}, \demph{Cauchy product}, and \demph{substitution product} of two species $\p$ and $\q$ are defined, respectively, by
\begin{align*}
(\p\splus\q)[I] &:= \p[I] \oplus \q[I] \\
(\p\stimes\q)[I] &:= \p[I]\otimes\q[I] \\
(\p\sdot\q)[I] &:= \bigoplus_{S\sqcup T = I} \p[S]\otimes \q[T] \\
(\p\scirc\q_+)[I] &:= \bigoplus_{X \vdash I} \p[X] \otimes \q(X),
\end{align*}
where: $(\p\scirc\q_+)[\emptyset]$ is defined to be $\p[\emptyset]$; and $\q(X)$ is shorthand for the unordered tensor product $\bigotimes_{A\in X} \q[A]$. 
\end{definition}

\begin{example}\label{ex:basic examples}
We use the operations above to rewrite some of the species introduced in Example \ref{ex:common examples} and introduce some new ones that will be important in what follows.
\begin{itemize}
\item Any species is built from its restrictions via infinite sum. Or, viewed the other way around, every species $\p$ has a \demph{canonical decomposition} $\p = \sum_{n\geq0} \p_n$ \cite[Ex.~1.3.6]{bergeron1998combinatorial}. (Above, we wrote $\p_+$ for $\sum_{n>0} \p_n$.) Also, we put $\p_{<r}:=\sum_{0\leq n<r}\p_n$ and $\p_{\geq r}:=\sum_{r\leq n}\p_n$.
\item The trivial species $\bf1$ is $\bfE_1$.
\item The species of {linear orders} $\bfL$, is, in fact, ${\bf1} \splus \bfE_1 \splus (\bfE_1\sdot\bfE_1) \splus \cdots = \sum_{n\geq0} \bfE_1^{{\tsdot} n}$.
\item The species of \demph{set partitions} $\bfPi$ is $\bfE\circ\bfE_+$. 
\end{itemize}
\end{example}

\begin{example}\label{ex:more examples}
Further shorthand and examples for substitution products $(\p\scirc\q_+)$. 
\begin{itemize}
\item Given $p\in \p[X]$ and elements $\{q_A \in \q[A] \mid A\in X\}$, we often write $p\otimes q_{(X)}$ for the simple tensor $p\otimes \bigotimes_{A\in X} q_A$. 

\item If $\p$ and $\q$ are linearized, a useful perspective views $\mathrm{p}\circ\mathrm{q}$-structures as $\mathrm{p}$-structures placed on a set of $\mathrm{q}$-structures. See the next two items as well as Figure \ref{fig:GoL-example}. 

\item The $(\rfL\circ\rfG)$-structures indexing a basis of $(\bfL\scirc\bfG_+)[\{a,b,c\}]$ come in a variety of types (depending on the cardinality of the underlying set partition), {\it e.g.},
\begin{align*}
X=\{\{a,b,c\}\}: \ \ & 
\bigl(\!\!\threegraphline[0]{a}{b}{c}{1}{1}{1}\!\!\bigr) \quad \ \ 
\bigl(\!\threegraphline[1]{b}{a}{c}{1}{0}{1}\!\bigr) \quad \ \ 
\bigl(\!\threegraphline[1]{a}{b}{c}{1}{0}{1}\!\bigr) \quad \ \ 
(a \ \ \ b\ \ \ c) \\ 
X=\{\{b,c\},\{a\}\}: \ \ & 
(a|b{-}c) \quad \ \ 
(a|b{\,\ }c) \quad \ \ 
(b{-}c|a) \quad \ \ 
(b{\,\ }c|a) \\
X=\{\{c\},\{a\},\{b\}\}: \ \ & 
(a|b|c) \quad \ \ (a|b|c) \quad \ \ 
(b|a|c) \quad \ \ (b|c|a) \quad \ \ 
(c|a|b) \quad \ \ (c|b|a)
\end{align*}
Above, ``$|$'' separates different elements (simple graphs) in the lists, and ``$-$'' indicates edges in the simple graphs.\footnote{{\bf N.B.} we also use ``$|$'' for assorted restriction operations.} 
 (See \eqref{eq:graphs-abc} for the complete list when $X=\{\{a,b,c\}\}$.) 

\item A $\rfG\circ\rfL_+$ structure is a simple graph on a vertex set of \emph{tuples} $l^{\,i} = (\ell^{\,i}_{\,1}\,|\,{\ell^{\,i}_{\,2}}\,|\,\cdots)$. {\it E.g.}, the $\rfG\circ\rfL_+$ structures on $I = \{a,b,c\}$ are the three-vertex graphs in \eqref{eq:graphs-abc} together with the graphs below.
\begin{align}
\notag
\twograph{(a)}{(b|c)}{1} &&
\twograph{(a)}{(c|b)}{1} &&
\twograph{(b)}{(a|c)}{1} &&
\twograph{(b)}{(c|a)}{1} &&
\twograph{(c)}{(a|b)}{1} &&
\twograph{(c)}{(b|a)}{1}
\\[.5ex]
\label{eq:graphs-GL}
\twograph{(a)}{(b|c)}{0} &&
\twograph{(a)}{(c|b)}{0} &&
\twograph{(b)}{(a|c)}{0} &&
\twograph{(b)}{(c|a)}{0} &&
\twograph{(c)}{(a|b)}{0} &&
\twograph{(c)}{(b|a)}{0}
\\[.25ex]
\notag
\onegraph{(a|b|c)} \ &&
\onegraph{(a|c|b)} \ &&
\onegraph{(b|a|c)} \ &&
\onegraph{(b|c|a)} \ &&
\onegraph{(c|a|b)} \ &&
\onegraph{(c|b|a)}
\end{align}
(The first and second row in \eqref{eq:graphs-GL} shows connected and disconnected graphs on two nodes, respectively; the third row shows graphs on one node.)
\end{itemize}
\end{example}

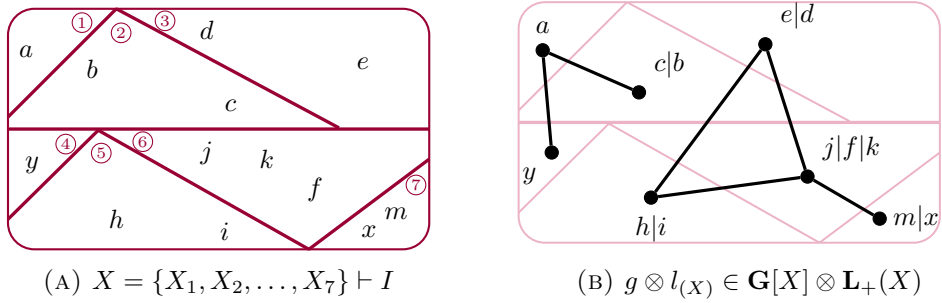
\begin{figure}[htb]
\centering
\begin{subfigure}[t]{.45\textwidth}
	\centering
    \begin{tikzpicture}[scale=.80, purp/.style={color=purple!80!black,inner sep=.75pt}]
    
      \draw[style=purp,line width=1.25pt,rounded corners=.5pt] (0,2.2) -- (1.8,4) -- (5.5,2.02); 
      \draw[style=purp,line width=1.25pt,rounded corners=.5pt] (0,0.5) -- (1.5,1.98) -- (5,0) -- (7,1.5); 
    
      \draw[style=purp,rounded corners=10,thick]
         (0,0) rectangle (7,4);
      \draw[style=purp,line width=1.25pt] (0,2) -- (7,2);
    
      \node[style=purp,circle,draw] (z1) at (1.23,3.77) {\footnotesize$\scriptstyle1$};
      \node[style=purp,circle,draw] (z2) at (1.88,3.60) {\footnotesize$\scriptstyle2$};
      \node[style=purp,circle,draw] (z3) at (2.61,3.80) {\footnotesize$\scriptstyle3$};
      \node[style=purp,circle,draw] (z4) at (0.96,1.75) {\footnotesize$\scriptstyle4$};
      \node[style=purp,circle,draw] (z5) at (1.55,1.60) {\footnotesize$\scriptstyle5$};
      \node[style=purp,circle,draw] (z6) at (2.24,1.79) {\footnotesize$\scriptstyle6$};
      \node[style=purp,circle,draw] (z7) at (6.78,1.07) {\footnotesize$\scriptstyle7$};
    
      \node (a) at (.3,3.3) {\small$a$};
      \node (b) at (1.4,3) {\small$b$};
      \node (d) at (3.3,3.65) {\small$d$};
      \node (e) at (5.9,3.1) {\small$e$};
      \node (c) at (3.7,2.4) {\small$c$};
      \node (h) at (1.8,.5) {\small$h$};
      \node (f) at (5.1,1) {\small$f$};
      \node (i) at (3.6,.3) {\small$i$};
      \node (j) at (3.3,1.6) {\small$j$};
      \node (k) at (4.3,1.5) {\small$k$};
      \node (m) at (6.45,.6) {\small$m$};
      \node (x) at (6,.3) {\small$x$};
      \node (y) at (.4,1.4) {\small$y$};
    \end{tikzpicture}
	\caption{$X = \{X_1, X_2, \ldots, X_7\} \vdash I$}
\end{subfigure}
\quad \ \ 
\begin{subfigure}[t]{.45\textwidth}
    \begin{tikzpicture}[scale=.80, 
    	vert/.style={draw,fill,circle,inner sep=0,minimum height=5pt},
    	purp/.style={color=purple!30!white,inner sep=.75pt}
    	]
    
      \draw[style=purp,line width=.75pt,rounded corners=.5pt] (0,2.2) -- (1.8,4) -- (5.5,2.02); 
      \draw[style=purp,line width=.75pt,rounded corners=.5pt] (0,0.5) -- (1.5,1.98) -- (5,0) -- (7,1.5); 
    
      \draw[style=purp,rounded corners=10,line width=.75pt]
         (0,0) rectangle (7,4);
      \draw[style=purp,line width=1.25pt] (0,2) -- (7,2);
    
      \node[style=vert, label=above:{\small$a$}] (a) at (.4,3.2) {};
      \node[style=vert, label=above right:{\small$c|b$}] (cb) at (2,2.5) {};
      \node[style=vert, label=above right:{\small$e|d$}] (ed) at (4.1,3.3) {};
      \node[style=vert, label=below left:{\small$y$}] (y) at (.55,1.5) {};
      \node[style=vert, label=below:{\small$h|i$}] (hi) at (2.2,.75) {};
      \node[style=vert, label=above right:{\small$j|f|k$}] (jfk) at (4.8,1.1) {};
      \node[style=vert, label={[label distance=-0.05cm]0:{\small$m|x$}}] (xm) at (6.0,.4) {};
     
      \draw[very thick] (y) -- (a) -- (cb);
      \draw[very thick] (xm) -- (jfk);
      \draw[very thick] (hi) -- (ed) -- (jfk) -- (hi);
    \end{tikzpicture}
	\caption{$g\otimes l_{(X)} \in \bfG[X]\otimes \bfL_+(X)$}
\end{subfigure}
\caption{A $(\rfG\circ\rfL_+)[I]$ structure is a simple graph on a set of linear orders $\{l_{X_i}\}_{X_i\in X}$.}
\label{fig:GoL-example}
\end{figure}

\subsection{Monoidal structures}
\label{sec:monoidal structures}

The category $\Spec$ is a braided monoidal category under Cauchy product, with $\mathbf1$ as the unit object (and trivial brading, $u\otimes v \mapsto v\otimes u$). 
 
Let $(\mathsf{C},\bullet)$ be a monoidal category with unit element $\mathrm{U}$. Recall a \demph{monoid} in $\mathsf{C}$ is an object $\mathrm B$ equipped with morphisms $\mu:\mathrm B\bullet \mathrm B \to \mathrm B$ and $\iota:\mathrm{U}\to \mathrm B$ satisfying certain associativity and unital axioms.
Similarly, $\mathrm B$ is a \demph{comonoid} in $\mathsf{C}$ if it's equipped with morphisms $\Delta:\mathrm B \to \mathrm B \bullet \mathrm B$ and $\epsilon:\mathrm B \to \mathrm{U}$ satisfying certain dual axioms.
Finally, $\mathrm B$ is a \demph{bimonoid} in (a braided monoidal category) $\mathsf{C}$ if these maps are compatible (briefly, $\Delta$ and $\epsilon$ are monoid-morphisms); and this bimonoid $\mathrm{B}$ is a \demph{Hopf monoid} in $\mathsf{C}$ if the identity morphism $\id$ on $\mathrm B$ is invertible in the convolution monoid, $(\End(\mathrm{B}), \boldsymbol{\ast})$. (Here, the convolution $f\mathrel{\boldsymbol{\ast}}g$ of two morphisms is defined as the composition $\mu \circ (f\bullet g)\circ \Delta$.) If $\id^{-1}$ exists, we call it the \demph{antipode}---ordinarily denoted ``$\mathrm{s}$.''

Below, we deal exclusively with connected species, so most  axioms involving the (co)units are trivial. We display the remaining axioms in our context.

\medskip\noindent
\emph{Associativity axiom for monoids.} For all decompositions $S\sqcup T$ of a finite set $I$, there exist morphisms $\mu_{S,T}:\b[S]\otimes \b[T] \to \b[I]$ making the following diagram commute:
\begin{equation}
\label{eq:associative}
\begin{tikzcd}
    \b[R]\otimes\b[S]\otimes\b[T] \arrow[rr, "\id_R\otimes \mu_{S,T}"]\arrow[d, "\mu_{R,S}\otimes \id_T"'] & & \b[R]\otimes\b[S\sqcup T] \arrow[d, "\mu_{R,S\sqcup T}"]\\
    \b[R\sqcup S]\otimes\b[T] \arrow[rr, "\mu_{R\sqcup S,T}"'] & & \b[R\sqcup S\sqcup T]
\end{tikzcd}
\end{equation}
%
\begin{remark}
A key word above is \emph{morphism}. That is, $\mu$ must also respects bijections. Specifically, if $\sigma:I \to J$ is a bijection with $S'=\sigma(S)$ and $T'=\sigma(T)$, then we need $\b[\sigma]\circ \mu_{S,T} = \mu_{S',T'}\circ \bigl( \b[\sigma\vert_{S}]\otimes \b[\sigma\vert_{T}]\bigr)$. Such checks are generally trivial and are omitted.
\end{remark}

\medskip\noindent
\emph{Coassociativity axiom for comonoids.} For all decompositions $S\sqcup T$ of a finite set $I$, there exist morphisms $\Delta_{I,J}: \b[I] \to \b[S]\otimes \b[T]$ making the following diagram commute:

\begin{equation}
\label{eq:coassociative}
\begin{tikzcd}
    \b[R\sqcup S\sqcup T] \arrow[rr, "\Delta_{R,S\sqcup T}"]\arrow[d, "\Delta_{R\sqcup S,T}"'] & & \b[R]\otimes\b[S\sqcup T] \arrow[d, "\id_S\otimes \Delta_{S,T}"]\\
    \b[R\sqcup S]\otimes\b[T] \arrow[rr, "\Delta_{R,S}\otimes \id_T"'] & & \b[R]\otimes\b[S]\otimes\b[T]
\end{tikzcd}
\end{equation}
%

\medskip\noindent
\emph{Compatibility axiom for binomoids.} 
Let $\tw:V\otimes W \to W\otimes V$ be the standard braiding on vector spaces over $\Bbbk$. Given two decompositions $S\sqcup T$ and $S'\sqcup T'$ of a finite set $I$, put 
\begin{gather}
\label{eq:ABCD}
	A=S\cap S', \ \ B = S\cap T', \ \ C = T\cap S', \qand D = T\cap T'. 
\end{gather}
Then the following diagram commutes:
\begin{equation}
\label{eq:compatible-delta}
\begin{tikzcd}
    \b[A]\otimes\b[B]\otimes\b[C]\otimes\b[D] \arrow[rr, "\id_A\otimes \tw \otimes \id_D"] & & \b[A]\otimes\b[C]\otimes\b[B]\otimes\b[D] \arrow[d, "\mu_{A,C}\otimes \mu_{B,D}"]\\
    \b[S]\otimes\b[T] \arrow[r, "\mu_{S,T}"'] \arrow[u, "\Delta_{A,B}\otimes\Delta_{C,D}"] & \b[I] \arrow[r,"\Delta_{S',T'}"'] & \b[S']\otimes\b[T'] 
\end{tikzcd}
\quad\mbox{\ }   
\end{equation}
%

\smallskip\noindent
\emph{Convolution-invertibility of $\id$.} Note that the identity morphism in $\End(\b)$ under convolution is $\iota\circ\epsilon$. So, for all finite sets $I$, we require morphisms $\mathrm{s}_I:\b[I] \to \b[I]$ making the following diagram (and its ``right invertible'' counterpart) commute:
\begin{equation}
\label{eq:antipode}
\begin{tikzcd}[column sep=tiny]
    {\displaystyle\bigoplus_{S\sqcup T}}\,\b[S]\!\otimes\!\b[T] \arrow[rr, "\mathrm{s}_S\otimes \id_T"] & & {\displaystyle\bigoplus_{S\sqcup T}}\, \b[S]\!\otimes\!\b[T] \arrow[d, start anchor={[yshift=1.5ex]south}, yshift=.25ex, "\oplus\mu_{S,T}"]\\
    \b[I] \arrow[r, "\epsilon_I"'] \arrow[u, end anchor={[yshift=1.5ex]south}, yshift=.25ex, "\oplus\Delta_{S,T}"'] & \mathbf1[I] \arrow[r,"\iota_I"'] & \b[I]\,.
\end{tikzcd}
\end{equation}

\begin{remark}
\label{rem:hopf-emptyset}
    (a). Note \eqref{eq:antipode} differs from \eqref{eq:associative}, \eqref{eq:coassociative}, and \eqref{eq:compatible-delta} in that we cannot compare maps one decomposition of $I$ at a time. Things aren't as bad as they seem: after \cite[Prop. 8.10]{aguiar2010monoidal}, $\b$ is a Hopf monoid if and only if $\b[\emptyset]$ is a Hopf algebra. This means it suffices to check \eqref{eq:antipode} in the case $I=\emptyset$. 

 (b).
If $\b$ is connected, then $\b[\emptyset]=\Bbbk$ is trivially a Hopf algebra. (So we make no further mention of antipodes in what follows.)
\end{remark}

\subsection{Three elementary examples of Hopf monoids: $\bfL$, $\bfE$, and $\bfG$}\label{sec:LEG definitions}

\mbox{}

\smallskip
\noindent
$\bfL$.~Given linear orders $l,l'$ of disjoint sets $S$ and $T$, respectively, write $l\cdot l'$ for the (unique) linear ordering of $S\sqcup T$ that places all of $S$ before all of $T$ and maintains the relative ordering within $S$ and $T$. That is, $(l\cdot l')\vert_S = l$ and $(l\cdot l')\vert_T = l'$. The two operations just indicated afford $\bfL$ the structure of Hopf monoid. Concretely, 
\begin{gather*}
\mu_{S,T}(l \otimes l') = l\cdot l'
\qand
\Delta_{S,T}(l) = l\bigr\vert_{S} \otimes l\bigr\vert_{T\,}.
\end{gather*}

\noindent$\bfE$.~The Hopf monoid of exponential species has $\bfE[I] = \Bbbk\{\ast_I\}$ for all finite $I$. The structure maps are:
\(\ds
\mu_{S,T}(\ast_S \otimes \ast_T) = \ast_{S\sqcup T}
\qand
\Delta_{S,T}(\ast_I) = \ast_{S} \otimes \ast_{T}.
\)

\smallskip
\noindent$\bfG$.~Let $g=(I,E)$ be a simple graph on vertex set $I$ with edge set $E$. Given $S\subseteq I$, we write $g\vert_S$ for the induced graph $(S, E\cap \binom{S}{2})$. Given simple graphs $g, g'$ on vertex sets $S,T$, write $g\sqcup g'$ for the graph $(S\sqcup T, E\sqcup T')$. These operations afford $\bfG$ the structure of Hopf monoid, 
\begin{gather}
\label{eq:G-structure}
	\mu_{S,T}(g\otimes g') = g \sqcup g'
\qand
	\Delta_{S,T}(g) = g\vert_S \otimes g\vert_{T\,},
\end{gather}
which lifts Schmitt's Hopf algebra \cite[Prop. 3.4]{schmitt1993hopf} to the level of species.


\subsection{Linearized bimonoids}

Given a linearized species $\bfB= \Bbbk \rfB$, we call $\bfB$ a \demph{linearized monoid} (respectively, \demph{linearized comonoid}) if its structure maps are linearizations of set maps $\mu_{S,T}:\rfB[S]\times\rfB[T] \to \rfB[S\sqcup T]$ (resp., of set maps $\Delta_{S,T}:\rfB[S\sqcup T] \to \rfB[S]\times \rfB[T]$). (We may view the latter as a pair of functions $(\lambda, \rho)$, {\it i.e.}, $b\xmapsto{\Delta} \lambda_S(b) \times \rho_T(b)$ for all $b\in\rfB[S\sqcup T]$.) A bimonoid $\bfB$ is a \demph{linearized bimonoid} if both structure maps $\mu$ and $\Delta$ are linearized.

\begin{remark} Most bimonoids in $(\Spec,\sdot)$ are not linearized; but neither are they scarce:
\begin{enumerate}
\item 
If $\b$ and $\d$ are linearized bimonoids, then so are $\b\sdot \d$ and $\b\stimes \d$. 
\item \label{itm:lin-bimonoid}
If $\b$ is a cocommutative linearized bimonoid and $\p$ is a positive linearized comonoid, then $\b\scirc\p$ is a linearized bimonoid ({\it cf.} \eqref{eq:bop-product}, \eqref{eq:bop-coproduct}, \& Theorem~\ref{th:bop}).
\end{enumerate}
\end{remark}

Marberg \cite{marberg2015linearization,marberg2016duality} has gone a long way towards classifying the self-dual linearized bimonoids. 
(If $\b=\Bbbk\rfB$ is commutative and cocommutative and \emph{strongly linearized} \cite[Def. 4.2.1]{marberg2015linearization}, then there exists a positive linearized comonoid $\q$ such that $\b\cong \bfE\scirc\q$. He also has a characterization in terms of a partial order on the set species $\rfB$ in this case.)
Linearized bimonoids also feature prominently in \cite{white2020cohen} and \cite{white2025chromatic}, where White introduces, respectively, notions of \emph{Cohen--Macaulay} and \emph{chromatic polynomial} for (combinatorial) Hopf monoids in species.

Note the cocommutative hypothesis in Item \ref{itm:lin-bimonoid}. In \cite[Sec. 8.7]{aguiar2010monoidal} we learn that the category of cocommutative linearized comonoids is equivalent to the category of species with restrictions introduced by Schmitt \cite{schmitt1993hopf}. And, moreover, cocommutative linearized bimonoids are the same as monoids in Schmitt's category.\footnote{Or \emph{coherent exponential species with restrictions}, in Schmitt's terminology \cite{schmitt1993hopf}.} This point of view will be useful in what follows, so we provide a few details here. 

A set species $\rfB$ is a \demph{species with restrictions} if it comes with a bundle of morphisms $\rho^I_U:\rfB[I] \to \rfB[U]$, one for each inclusion $U \subseteq I$ of finite subsets (including $\emptyset \subseteq I$), satisfying certain compatibility conditions: $\rho_I^I = \id$; and $\rho_U^V\circ \rho_V^I = \rho_U^I$.

The restriction morphisms yield a (cocommutative, linearized) coproduct for $\b = \Bbbk\rfB$ by putting $\Delta_{S,T}(b) = \rho_S^I(b) \otimes \rho_T^I(b)$ for all $b\in\rfB[I]$, {\it cf.} \cite[Prop. 8.29]{aguiar2010monoidal}. We sometimes write $b\vert_S$ in place of $\rho_S^I(b)$ to simplify notation. If $\b$ is a linearized bimonoid with this coproduct, then the compatibility axiom translates to the following \demph{coherence property} at the level of set species: for all finite sets $U \subseteq I = S \sqcup T$, the following diagram commutes, {\it cf.} \cite[Prop. 8.31]{aguiar2010monoidal}.
\begin{equation}
\label{eq:rescriction-coherence}
\begin{tikzcd}[column sep=8em]
\rfB[S]\times \rfB[T] \arrow[d, "\mu_{S,T}"'] \arrow[r, "\rho^S_{S\cap U} \times \rho^T_{T\cap U}"] & 
	\rfB[S\cap U]\times \rfB[T\cap U] \arrow[d, "\mu_{S\cap U, T\cap U}"]
	\\
\rfB[I] \arrow[r, "\rho^I_{U}"'] & \rfB[U]
\end{tikzcd}
\end{equation}

The Hopf monoids $\bfL$, $\bfE$, and $\bfG$ are all cocommutative linearized bimonoids. 

\subsection{The functors $\cfT(\mhyphen)$ and $\cfS(\mhyphen)$, briefly}
\label{sec:two-functors}

Let $\p$ be a positive comonoid. In \cite[Ch. 11]{aguiar2010monoidal} one finds a Hopf monoid structure $\cfT(\p)$ on $\bfL\scirc\p$ satisfying the following universal property: $\p$ embeds in $\cfT(\p)_+$ via a comonoid map $\eta_{\p}$; and for all comonoid maps $\zeta:\p \to \h_+$ to (the positive part of) a Hopf monoid $\h$, there exists a Hopf monoid map $\hat\zeta:\cfT(\p)\to\h$ making the following diagram commute.
	\[
	\begin{tikzcd}
		\cfT(\p)_+ \arrow[rr, "\hat{\zeta}_{+}"]  &  & \h_+ \\
		& \p \arrow[ul, "\eta(\p)"] \arrow[ur, "\zeta"']& 
	\end{tikzcd}
	\]

The monoid structure is simple enough: if $X\vdash S$ and $Y\vdash T$, then $\mu_{S,T}$ maps the pair $l\otimes p_{(X)}$ and $l'\otimes p_{(Y)}$ to the element $(l\cdot l')\otimes p_{(X\sqcup Y)}$. We leave a careful description of the comonoid structure for Section \ref{sec:bop-bimonoid-proof}. (Or the reader may consult \cite[Sec. 11.2.4]{aguiar2010monoidal}, which our construction closely follows.) Roughly: given $l\otimes p_{(X)}$, split $p_{(X)}$ according to $\Delta^{\p}$ and restrict $l$ to the partitions of $S$ and $T$ that $X$ induces. 
Two examples in $\cfT(\bfG_+)$ with $X=\bigl\{\{d\},\{b\},\{a,c,e\}\bigr\}$ follow:
\[
(d|b|c{-}e{\,\ }a)\xmapsto{\Delta_{cde,ab}} (d|c{-}e) \otimes (b|a)
\ \ \qand\ \ 
(d|b|c{-}e{\,\ }a)\xmapsto{\Delta_{bed,ac}} (d|b|e) \otimes (c{\, \ }a).
\]

The Hopf monoid $\cfS(\p)$ built on $\bfE\scirc\p$ is defined similarly. 
It satisfies nearly the same universal property as $\cfT(\p)$, adding only the assumption that $\h$ is commutative. 


\section{The Hopf monoid $\tee{\b}{\p}$}

For fixed positive comonoid $\p$, the constructions $\cfT(\p)$ and $\cfS(\p)$ in \cite{aguiar2010monoidal} were made without explicit reference to species with restrictions. Instead, the proofs that they are Hopf monoids rely on the existence of certain colax-colax adjunctions. (This might lead one to imagine the (co, bi)monoid structures on $\bfL\scirc\p$ and $\bfE\scirc\p$ are very special. Not quite so.) In this section, we extend these constructions from $\bfL$ and $\bfE$ to any cocommutative linearized bimonoid $\b$. We let $\tee{\b}{\p}$ denote the resulting Hopf monoid structure on $\b\scirc\p$.

\subsection{A monoid and comonoid structure on $\b\scirc\p$}
\label{sec:bop-bimonoid-proof}

\mbox{}

\medskip\noindent
\underline{\emph{A monoid structure on $\b\scirc\p$.}} 
Suppose $\b$ is a monoid and $\p$ is positive. Given partitions $X\vdash S$ and $Y\vdash T$, and simple tensors $b\otimes p_{(X)} \in \b[X]\otimes \p(X)$ and $b'\otimes p_{(Y)} \in \b[Y]\otimes \p(Y)$, define the mapping $\mu_{S,T}$ via
\begin{gather}\label{eq:bop-product}
\bigl(b\otimes p_{(X)}\bigr) \otimes \bigl(b'\otimes p_{(Y)}\bigr) \ \xmapsto{\mu_{S,T}} \ \mu_{X,Y}^{\b}(b\otimes b') \otimes p_{(X\sqcup Y)}\,.
\end{gather}
%
%

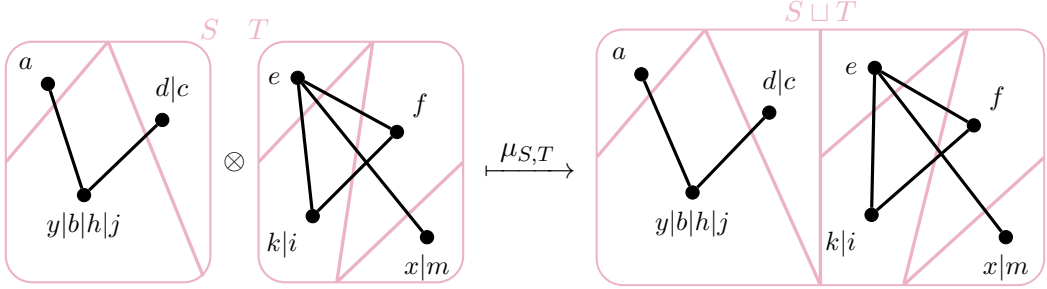
\begin{figure}[htb]
\centering
    \begin{tikzpicture}[scale=.80, 
    		vert/.style={draw,fill,circle,inner sep=0,minimum height=5pt},
			purp/.style={color=purple!30!white,inner sep=.75pt}
			]
    
      \draw[style=purp,rounded corners=10,thick] (-.3,0) rectangle (3.1,4);
      \draw[style=purp,rounded corners=10,thick] (3.9,0) rectangle (7.3,4);
    
      \draw[style=purp,line width=1.25pt,rounded corners=1pt] (-.3,2) -- (1.4,4) -- (2.98,.1);
      \draw[style=purp,line width=1.25pt,rounded corners=1pt] (3.9,2) -- (5.8,4) -- (5.2,0) -- (7.3,2);
  \node[style=purp,above] () at (3.5,4) {$S \ \ \ T$};
    
    
      \node[style=vert, label=above left:{\small$a$}] (a) at (.4,3.3) {};
      \node[style=vert, label=above:{\ \ \small$d|c$}] (dc) at (2.3,2.7) {};
      \node[style=vert, label=below:{\small$y|b|h|j$}] (ybhj) at (1,1.45) {};

      \node[style=vert, label=above right:{\small$f$}] (f) at (6.2,2.5) {};
      \node[style=vert, label=left:{\small$e$}] (e) at (4.55,3.4) {};
      \node[style=vert, label=below left:{\small$k|i$}] (ki) at (4.8,1.1) {};
      \node[style=vert, label=below:{\small$x|m$}] (xm) at (6.7,.75) {};
     
      \draw[very thick] (a) -- (ybhj) -- (dc);
      \draw[very thick] (xm) -- (e);
      \draw[very thick] (f) -- (ki) -- (e) -- (f);

	\node (X) at (3.5,2) {$\otimes$};
    \end{tikzpicture}
\ \raisebox{9ex}{$\displaystyle \xmapsto{\displaystyle\,\,\mu_{S,T}\,\,}$} \ 
    \begin{tikzpicture}[scale=.85, 
    	vert/.style={draw,fill,circle,inner sep=0,minimum height=5pt},
    	purp/.style={color=purple!30!white,inner sep=.75pt}
    	]
    
      \draw[style=purp,line width=1.25pt,rounded corners=1pt] (0,2) -- (1.7,4) -- (3.5,0);
      \draw[style=purp,line width=1.25pt,rounded corners=1pt] (3.5,2) -- (5.8,4) -- (4.8,0) -- (7,2);

	\draw[style=purp,line width=1.25pt] (3.5,0) -- (3.5,4);
    \node[style=purp,above] () at (3.5,4.1) {$S \sqcup T$};

      \draw[style=purp,rounded corners=10,line width=.75pt]
         (0,0) rectangle (7,4);
    
      \node[style=vert, label=above left:{\small$a$}] (a) at (.7,3.3) {};
      \node[style=vert, label=above:{\ \ \small$d|c$}] (dc) at (2.7,2.7) {};
      \node[style=vert, label=below:{\small$y|b|h|j$}] (ybhj) at (1.5,1.45) {};

      \node[style=vert, label=above right:{\small$f$}] (f) at (5.9,2.5) {};
      \node[style=vert, label=left:{\small$e$}] (e) at (4.35,3.4) {};
      \node[style=vert, label=below left:{\small$k|i$}] (ki) at (4.3,1.1) {};
      \node[style=vert, label=below:{\small$x|m$}] (xm) at (6.4,.75) {};
     
      \draw[very thick] (a) -- (ybhj) -- (dc);
      \draw[very thick] (xm) -- (e);
      \draw[very thick] (f) -- (ki) -- (e) -- (f);
    \end{tikzpicture}
\caption{The product on $\bfG\scirc\bfL_+$.}
\label{fig:GoL-product}
\end{figure}

\begin{proof}[Proof of monoid structure]
Checking associativity is straightforward, following from the associativity of $\mu^{\b}$ and of set union. (Given $I = R\sqcup S \sqcup T$, $X\vdash R$, $Y\vdash S$, and $Z\vdash T$. Beginning with the simple tensor $b_X\otimes p_{(X)}\otimes b_Y\otimes p_{(T)}\otimes b_Z\otimes p_{(Z)}$, and taking either route around the square in \eqref{eq:associative}, we get
$
(b_X \cdot b_Y \cdot b_Z)\otimes p_{(X\sqcup Y \sqcup Z)}
$.)

Again, we omit all (co)unit checks as they are trivial in our context.
Conclude that $(\b\scirc\p_+, \mu)$ is a monoid. 
\end{proof}

\medskip\noindent
\underline{\emph{A comonoid structure on $\b\scirc\p$.}}
Suppose $\b$ is a connected cocommutative linearized comonoid and $\p$ is a positive comonoid. We need some notation before defining the coproduct for $\b\scirc\p$. Given a partition $X=\{X_1,X_2,\ldots\}$ of $I$, and a nonempty subset $T \subseteq I$, write
\begin{gather}
\label{eq:partition-restriction}
X^T:=\{X_i \in X \mid X_i \cap T \neq\emptyset\}
\qand
X_T:=\{X_i\cap T \in X \mid X_i \cap T \neq\emptyset\}.
\end{gather}
See Figure \ref{fig:restriction-notation} (together with Fig. \ref{fig:GoL-example}) for an example.
\begin{figure}[htb]
\centering
\begin{minipage}[b]{.42\textwidth}{%
\begin{align*}
X &
	=\{a \ / \ b,c\ / \ d,e\ / \ y\ / \ h,i\ / \ j,k,f\ / \ x,m\} \quad\mbox{\ }
& 
T &= \{c,e,i,k,f,x,m\}
\\
X^T &
	= \{\,b,c \ / \ d,e \ / \ h,i \ / \ j,k,f \ / \ x,m \,\} 
&
X_T &
	= \{\,c \ / \ e \ / \ i \ / \ k,f \ / \ x,m\,\} \\
\hphantom{Y}
\end{align*}
}\end{minipage}
\vskip-.2in

\begin{tikzpicture}[scale=.85, 
	vert/.style={draw,fill,circle,inner sep=0,minimum height=5pt},
	purp/.style={color=purple!40!white,inner sep=.75pt}
	]

  \draw[style=purp,line width=.75pt,rounded corners=.5pt] (0,2.2) -- (1.8,4) -- (5.5,2.02); 
  \draw[style=purp,line width=.75pt,rounded corners=.5pt] (0,0.5) -- (1.5,1.98) -- (5,0) -- (7,1.5); 

  \draw[style=purp,rounded corners=10,line width=.75pt]
     (0,0) rectangle (7,4);
  \draw[style=purp,line width=1.25pt] (0,2) -- (7,2);

  \node[style=vert, label=above right:{\small$c|b$}] (cb) at (2,2.5) {};
  \node[style=vert, label=above right:{\small$e|d$}] (ed) at (4.1,3.3) {};
  \node[style=vert, label=below:{\small$h|i$}] (hi) at (2.2,.75) {};
  \node[style=vert, label=above right:{\small$j|f|k$}] (jkf) at (4.8,1.1) {};
  \node[style=vert, label={[label distance=-0.05cm]0:{\small$m|x$}}] (xm) at (6.0,.4) {};
 
  \draw[very thick] (xm) -- (jkf);
  \draw[very thick] (hi) -- (ed) -- (jkf) -- (hi);

\node (subcaption) at (3.5,-.6) {$g\vert_{X^{T}}\otimes l_{(X^{T})} \in \bfG[X^T]\otimes \bfL_+(X^{T})$};
\end{tikzpicture}
\quad
\begin{tikzpicture}[scale=.85, 
	vert/.style={draw,fill,circle,inner sep=0,minimum height=5pt},
	purp/.style={color=purple!40!white,inner sep=.75pt}
	]

  \draw[style=purp,line width=.75pt,rounded corners=.5pt] (0,2.2) -- (1.8,4) -- (5.5,2.02); 
  \draw[style=purp,line width=.75pt,rounded corners=.5pt] (0,0.5) -- (1.5,1.98) -- (5,0) -- (7,1.5); 

  \draw[style=purp,rounded corners=10,line width=.75pt]
     (0,0) rectangle (7,4);
  \draw[style=purp,line width=1.25pt] (0,2) -- (7,2);

  \node[style=vert, label=above right:{\small$c$}] (cb) at (2,2.5) {};
  \node[style=vert, label=above right:{\small$e$}] (ed) at (4.1,3.3) {};
  \node[style=vert, label=below:{\small$h|i$}] (i) at (2.2,.75) {};
  \node[style=vert, label=above right:{\small$f|k$}] (jkf) at (4.8,1.1) {};
  \node[style=vert, label={[label distance=-0.05cm]0:{\small$m|x$}}] (xm) at (6.0,.4) {};
 
  \draw[very thick] (xm) -- (jkf);
  \draw[very thick] (hi) -- (ed) -- (jkf) -- (hi);

\node (subcaption) at (3.5,-.6) {$\sigma^{X^{T}}_{X_{T}}\bigl(g\vert_{X^{T}}\bigr)\otimes l_{(X_{T})} \in (\bfG\scirc\bfL_+)[T]$};
\end{tikzpicture}
\caption{The notations $X^T$ and $X_T$ for $T\subseteq I$ and $X\vdash I$, and their use in the coproduct for $\bfG\scirc\bfL_+$. (See Fig. \ref{fig:GoL-example} for the graph $g$.)}
\label{fig:restriction-notation}
\end{figure}
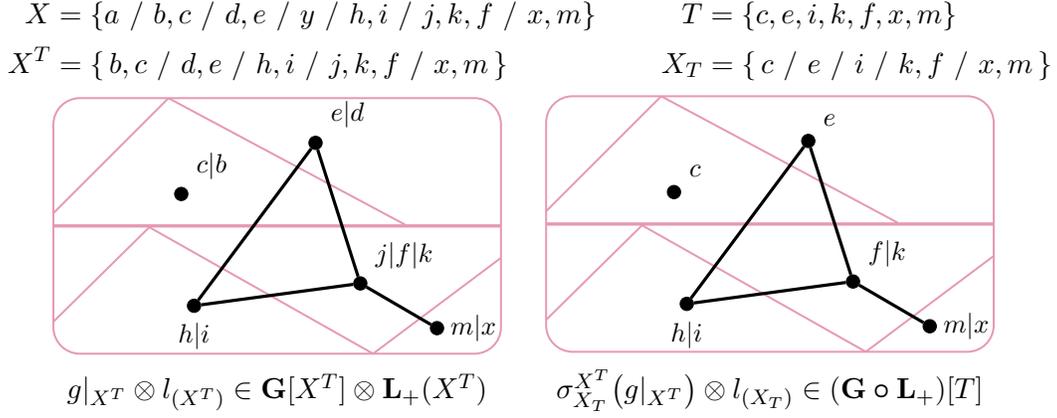
Given a decomposition $R \sqcup S \sqcup T = I$, it's plain to see that $(X^{S\sqcup T})^T = X^T$ and $(X_{S\sqcup T})_T = X_T$, and likewise for other combinations of $R,S,T$. To lighten notation, we sometimes write, {\it e.g.}, $X^{ST}$ in place of $X^{S\sqcup T}$.

Note that the sets $X^T$ and $X_T$ have the same cardinality---with the latter being a partition of $T$. A bijection: take each element $X_i \in X^T$ to the unique subset $Y_i \in X_T$ that it contains. Write $\sigma^{X^T}_{X_T}:\b[X^T] \xrightarrow{\,\simeq\,} \b[X_T]$ for the corresponding isomorphism provided by the species axioms; it will be useful farther below. To further lighten notation, we write $b\vert^{\downarrow}_{X_T}$ for the composition $\sigma^{X^T}_{X_T}\circ\rho^X_{X^T}(b)$ applied to some $b\in \b[X]$.

Next, suppose $X\vdash I$ and $I = S\sqcup T$. Build a decomposition for each block of $X$ via $X_i = S_i\sqcup T_i = (X_i \cap S)\sqcup (X_i \cap T)$. Note that $X_S$ (a partition of $S$, remember) is the union of the nonempty blocks $S_i$ just created, and likewise for $X_T$. Using the (positive) coproduct for $\p$, we define a map $\tilde\Delta_{(S,T)}$ from $\p(X)$ to $\p(X_S) \otimes \p(X_T)$ as follows. 

\noindent
\textbf{Step 1.} For each tensor factor $p_{X_i}$ of a given $p_{(X)} \in \p(X)$, perform the mapping
\[
	p_{X_i} \mapsto \begin{cases}
		p_{X_i} & \hbox{if }S_i=\emptyset \hbox{ or } T_i=\emptyset, \\
		\Delta^\p_{S_i,T_i}(p_{X_i}) & \hbox{otherwise.}
		\end{cases}
\]
Here, the images in the first case are viewed as belonging to $\p[S_i]$ or $\p[T_i]$, as appropriate; and in the second case as belonging to $\p[S_i]\otimes \p[T_i]$.

\smallskip
\noindent
\textbf{Step 2.} De-shuffle the resulting tensor factors so the $\p$-structures on the $S_i$ precede those on the $T_i$. The result is an element of $\p(X_S)\otimes \p(X_T)$, as needed. Using a modified Sweedler notation, we write that element as $\sum_{(p_{(X)})} p_{(X_S)}\otimes p_{(X_T)}$.

\medskip
We are ready to define the coproduct for $\b\scirc\p$.
%
Given $b\otimes p_{(X)} \in \b[X]\otimes \p(X)$, we write
\begin{gather}\label{eq:bop-coproduct}
    \Delta_{S,T}(b\otimes p_{(X)}) = \sum_{(p_{(X)})} \bigl(b\vert^{\downarrow}_{X_S} \otimes p_{(X_S)}\bigr) \otimes \bigl(b\vert^{\downarrow}_{X_T} \otimes p_{(X_T)}\bigr).
\end{gather}
%

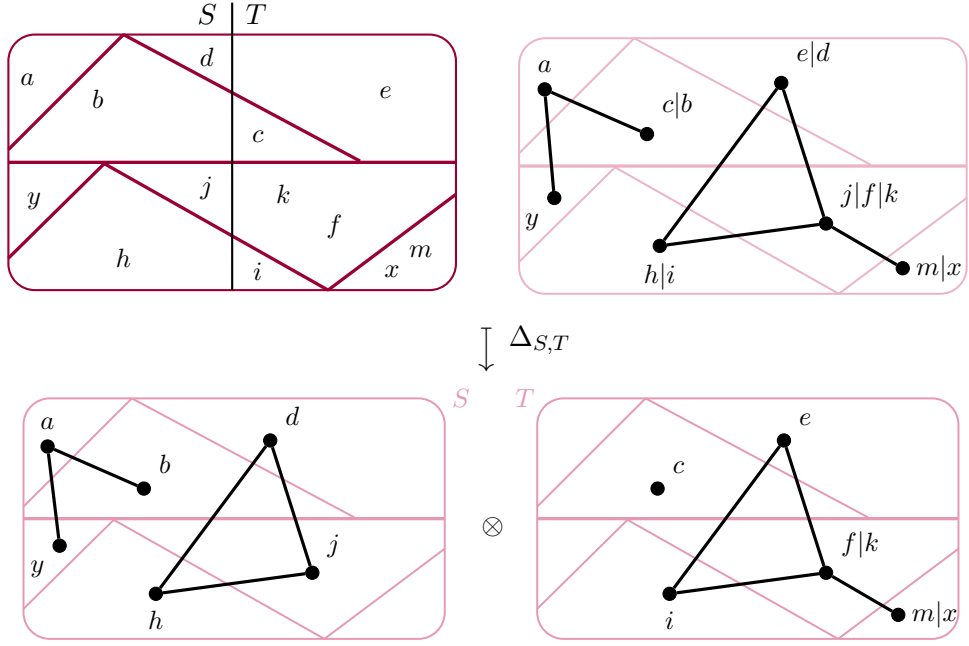
\begin{figure}[htb]
\centering
    \begin{tikzpicture}[scale=.85, purp/.style={color=purple!80!black,inner sep=.75pt}]
    
      \draw[style=purp,line width=1.25pt,rounded corners=.5pt] (0,2.2) -- (1.8,4) -- (5.5,2.02); 
      \draw[style=purp,line width=1.25pt,rounded corners=.5pt] (0,0.5) -- (1.5,1.98) -- (5,0) -- (7,1.5); 
    
      \draw[style=purp,rounded corners=10,thick]
         (0,0) rectangle (7,4);
      \draw[style=purp,line width=1.25pt] (0,2) -- (7,2);
  \draw[thick] (3.5,0) -- (3.5,4.5);
  \node[above] () at (3.5,4) {$S \ \ \ T$};
    
    
      \node (a) at (.3,3.3) {\small$a$};
      \node (b) at (1.4,3) {\small$b$};
      \node (d) at (3.1,3.65) {\small$d$};
      \node (e) at (5.9,3.1) {\small$e$};
      \node (c) at (3.9,2.4) {\small$c$};
      \node (h) at (1.8,.5) {\small$h$};
      \node (f) at (5.1,1) {\small$f$};
      \node (i) at (3.9,.3) {\small$i$};
      \node (j) at (3.1,1.6) {\small$j$};
      \node (k) at (4.3,1.5) {\small$k$};
      \node (m) at (6.45,.6) {\small$m$};
      \node (x) at (6,.3) {\small$x$};
      \node (y) at (.4,1.4) {\small$y$};
    \end{tikzpicture}
\quad \ \ 
\raisebox{-0.6ex}{%
    \begin{tikzpicture}[scale=.85, 
    	vert/.style={draw,fill,circle,inner sep=0,minimum height=5pt},
    	purp/.style={color=purple!30!white,inner sep=.75pt}
    	]
    
      \draw[style=purp,line width=.75pt,rounded corners=.5pt] (0,2.2) -- (1.8,4) -- (5.5,2.02); 
      \draw[style=purp,line width=.75pt,rounded corners=.5pt] (0,0.5) -- (1.5,1.98) -- (5,0) -- (7,1.5); 
    
      \draw[style=purp,rounded corners=10,line width=.75pt]
         (0,0) rectangle (7,4);
      \draw[style=purp,line width=1.25pt] (0,2) -- (7,2);
    
      \node[style=vert, label=above:{\small$a$}] (a) at (.4,3.2) {};
      \node[style=vert, label=above right:{\small$c|b$}] (cb) at (2,2.5) {};
      \node[style=vert, label=above right:{\small$e|d$}] (ed) at (4.1,3.3) {};
      \node[style=vert, label=below left:{\small$y$}] (y) at (.55,1.5) {};
      \node[style=vert, label=below:{\small$h|i$}] (hi) at (2.2,.75) {};
      \node[style=vert, label=above right:{\small$j|f|k$}] (jfk) at (4.8,1.1) {};
	  \node[style=vert, label={[label distance=-0.05cm]0:{\small$m|x$}}] (xm) at (6.0,.4) {};
     
      \draw[very thick] (y) -- (a) -- (cb);
      \draw[very thick] (xm) -- (jfk);
      \draw[very thick] (hi) -- (ed) -- (jfk) -- (hi);
    \end{tikzpicture}
}

\hphantom{$\Delta_{S,T}$}\rotatebox{-90}{$\longmapsto$}\ \ \raisebox{-\height}{$\Delta_{S,T}$}

\medskip
\begin{tikzpicture}[scale=.80, 
	vert/.style={draw,fill,circle,inner sep=0,minimum height=5pt},
	purp/.style={color=purple!40!white,inner sep=.75pt}
	]

      \draw[style=purp,line width=.75pt,rounded corners=.5pt] (0,2.2) -- (1.8,4) -- (5.5,2.02); 
      \draw[style=purp,line width=.75pt,rounded corners=.5pt] (0,0.5) -- (1.5,1.98) -- (5,0) -- (7,1.5); 
    
      \draw[style=purp,rounded corners=10,line width=.75pt]
         (0,0) rectangle (7,4);
      \draw[style=purp,line width=1.25pt] (0,2) -- (7,2);
    
      \node[style=purp, right] (S) at (7,4) {\,\small$S$};
    
      \node[style=vert, label=above:{\small$a$}] (a) at (.4,3.2) {};
      \node[style=vert, label=above right:{\small$b$}] (fcb) at (2,2.5) {};
      \node[style=vert, label=above right:{\small$d$}] (ed) at (4.1,3.3) {};
      \node[style=vert, label=below left:{\small$y$}] (y) at (.6,1.55) {};
      \node[style=vert, label=below:{\small$h$}] (hi) at (2.2,.75) {};
      \node[style=vert, label=above right:{\small$j$}] (kj) at (4.8,1.1) {};
     
      \draw[very thick] (y) -- (a) -- (fcb);
      \draw[very thick] (hi) -- (ed) -- (kj) -- (hi);
\end{tikzpicture}
\raisebox{8.5ex}{$\otimes$}
\begin{tikzpicture}[scale=.80, 
	vert/.style={draw,fill,circle,inner sep=0,minimum height=5pt},
	purp/.style={color=purple!40!white,inner sep=.75pt}
	]

      \draw[style=purp,line width=.75pt,rounded corners=.5pt] (0,2.2) -- (1.8,4) -- (5.5,2.02); 
      \draw[style=purp,line width=.75pt,rounded corners=.5pt] (0,0.5) -- (1.5,1.98) -- (5,0) -- (7,1.5); 
    
      \draw[style=purp,rounded corners=10,line width=.75pt]
         (0,0) rectangle (7,4);

      \draw[style=purp,line width=1.25pt] (0,2) -- (7,2);
    
      \node[style=purp, left] (T) at (0,4) {\small$T$};
    
      \node[style=vert, label=above right:{\small$c$}] (fcb) at (2,2.5) {};
      \node[style=vert, label=above right:{\small$e$}] (ed) at (4.1,3.3) {};
      \node[style=vert, label=below:{\small$i$}] (hi) at (2.2,.75) {};
      \node[style=vert, label=above right:{\small$f|k$}] (kj) at (4.8,1.1) {};
	  \node[style=vert, label={[label distance=-0.05cm]0:{\small$m|x$}}] (xm) at (6.0,.4) {};
     
      \draw[very thick] (xm) -- (kj);
      \draw[very thick] (hi) -- (ed) -- (kj) -- (hi);
\end{tikzpicture}
\caption{The coproduct on $\bfG\scirc\bfL_+$.}
\label{fig:GoL-coproduct}
\end{figure}

\begin{proof}[Proof of comonoid structure]
Given $I = R\sqcup S \sqcup T$ and $X \vdash I$. Coassociativity of $\Delta^\p$ guarantees the same for $\tilde\Delta_{(S,T)}$. Specifically, the reader can verify that 
\[
	\tilde\Delta_{(R,S)} \otimes \id_{T} \circ \tilde\Delta_{(R\sqcup S,T)} 
	=
	\id_R \otimes \tilde\Delta_{(S,T)} \circ \tilde\Delta_{(R,S\sqcup T)}. 
\]
The behavior of the $\b$ factors in \eqref{eq:coassociative} is more subtle. (Figure \ref{fig:bop-coassociativity} may offer a helpful illustration as we investigate coassociativity.)
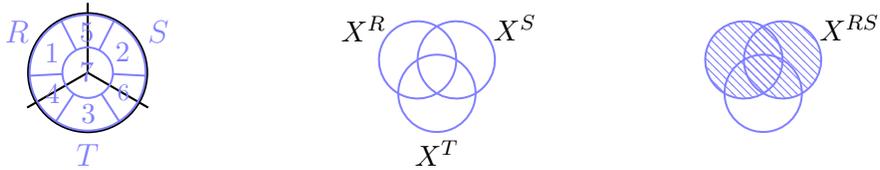
\begin{figure}[!htb]
\centering
\tikzset{every picture/.style={line width=0.75pt,scale=.75}} 
\begin{tikzpicture}[x=0.75pt,y=0.75pt,yscale=-1,xscale=1]

\draw[color=black,thick] (0,0) circle (40);
\draw[black,thick,rotate=-90] (0,0) -- (47,0); 
\draw[black,thick,rotate=30] (0,0) -- (47,0); 
\draw[black,thick,rotate=150] (0,0) -- (47,0); 

\draw[color=blue!50!white] (0,0) circle (17);
\draw[color=blue!50!white] (0,0) circle (39);
\draw[color=blue!50!white,rotate=-63,shift={(17,0)}] (0,0) -- (22,0); 
\draw[color=blue!50!white,rotate=-117,shift={(17,0)}] (0,0) -- (22,0); 
\draw[color=blue!50!white,rotate=3,shift={(17,0)}] (0,0) -- (22,0); 
\draw[color=blue!50!white,rotate=57,shift={(17,0)}] (0,0) -- (22,0); 
\draw[color=blue!50!white,rotate=123,shift={(17,0)}] (0,0) -- (22,0); 
\draw[color=blue!50!white,rotate=177,shift={(17,0)}] (0,0) -- (22,0); 

\draw[color=blue!50!white,rotate=-150,shift={(55,0)}] node at (0,0) {\large $R$};
\draw[color=blue!50!white,rotate=-30,shift={(55,0)}] node at (0,0) {\large $S$};
\draw[color=blue!50!white,rotate=90,shift={(55,0)}] node at (0,0) {\large $T$};

\draw[color=blue!50!white,rotate=-150,shift={(28,0)}] node (1) at (-.5,-.5) { $1$};
\draw[color=blue!50!white,rotate=-30,shift={(28,0)}] node (2) at (-.5,-.5) { $2$};
\draw[color=blue!50!white,rotate=90,shift={(28,0)}] node (3) at (-.5,-.5) { $3$};

\draw[color=blue!50!white,rotate=150,shift={(28,0)}] node (4) at (-.5,-.5) { $4$};
\draw[color=blue!50!white,rotate=-90,shift={(28,0)}] node (5) at (-.5,-.5) { $5$};
\draw[color=blue!50!white,rotate=30,shift={(28,0)},font=\footnotesize,] node (6) at (-.5,-.5) { $6$};
\draw[color=blue!50!white,] node (7) at (-.5,-.5) { $7$};
\end{tikzpicture}
\hskip.75in
\begin{tikzpicture}[x=0.75pt,y=0.75pt,yscale=-1,xscale=1]
\draw[color=blue!50!white,rotate=-30,shift={(15,0)}] (0,0) circle (26);
\draw[color=blue!50!white,rotate=-150,shift={(15,0)}] (0,0) circle (26);
\draw[color=blue!50!white,rotate=90,shift={(15,0)}] (0,0) circle (26);

\draw[color=black,rotate=-150,shift={(57,0)}] node at (0,0) {$X^R$};
\draw[color=black,rotate=-30,shift={(60,1)}] node at (0,0) {$X^S$};
\draw[color=black,rotate=90,shift={(55,0)}] node at (0,0) {$X^T$};
\end{tikzpicture}
\hskip.75in
\begin{tikzpicture}[x=0.75pt,y=0.75pt,yscale=-1,xscale=1]
\draw[color=blue!50!white,pattern=north west lines, pattern color=blue!50!white,rotate=-30,shift={(15,0)}] (0,0) circle (26);
\draw[color=blue!50!white,pattern=north west lines, pattern color=blue!50!white,rotate=-150,shift={(15,0)}] (0,0) circle (26);
\draw[color=blue!50!white,rotate=90,shift={(15,0)}] (0,0) circle (26);

\draw[color=black,rotate=-30,shift={(62,2)}] node at (0,0) {$\,\,X^{RS}$};
 \draw[color=black,rotate=90,shift={(55,0)}] node at (0,0) {\vphantom{$X^T$}};
\end{tikzpicture}
\caption{At left: a decomposition $I=R\sqcup S \sqcup T$ interacting with a partition $X\vdash I$ (into seven blocks). At center and right: illustrating the sets $X^R, X^{RS}$, etc.}
\label{fig:bop-coassociativity}
\end{figure}

Note that $X_{RS}$ is a partition of $R\sqcup S$; this affords us the notations $(X_{RS})^S$ and $(X_{RS})_S$, as in \eqref{eq:partition-restriction}. The hypotheses on $\b$ give us restrictions $\rho^X_{U}$ and $\rho^X_{V}$ that are morphisms in $(\Spec,\sdot)$. That is, we have
\begin{gather}
\label{eq:restriction-is-a-morphism}
	\rho^{X_{RS}}_{(X_{RS})^{S}}\sigma_{X_{RS}}^{X^{RS}} = 	
	\sigma_{(X_{RS})^S}^{(X^{RS})^S}\rho^{X^{RS}}_{(X^{RS})^S} =
	\sigma_{(X_{RS})^S}^{X^S}\rho^{X^{RS}}_{X^S};
\end{gather}
and likewise for other combinations of $R, S, T$. 

We temporarily relax the cocommutative assumption on $\b$, writing its comonoid structure as $b \mapsto \lambda^I_U(b) \otimes \rho^I_{V}(b)$. Proceeding counterclockwise in \eqref{eq:coassociative} and focusing on the $
\b$-factors yields
\begin{align*}
& \ 
\left[\bigl(\sigma_{(X_{RS})_R}^{(X_{RS})^R}\lambda^{X_{RS}}_{(X_{RS})^R} \otimes \sigma_{(X_{RS})_S}^{(X_{RS})^S}\rho^{X_{RS}}_{(X_{RS})^S}\bigr) \otimes \id_{X_T}\right]\circ\bigl(\sigma_{X_{RS}}^{X^{RS}}\lambda^X_{X^{RS}} \otimes \sigma_{X_{T}}^{X^{T}}\rho^X_{X^{T}}\bigr) \\
=& \ 
\left[\bigl(\sigma_{(X_{RS})_R}^{(X_{RS})^R}\sigma_{(X_{RS})^R}^{X^R}\lambda^{X^{RS}}_{X^R}\lambda^{X}_{X^{RS}} \otimes \sigma_{(X_{RS})_S}^{(X_{RS})^S}\sigma_{(X_{RS})^S}^{X^S}\rho^{X^{RS}}_{X^S}\lambda^{X}_{X^{RS}}\bigr) \otimes \id_{X_T}\right]\circ\bigl(\id_{X} \otimes \sigma_{X_{T}}^{X^{T}}\rho^X_{X^{T}}\bigr) \\
=& \ \sigma_{X_{R}}^{X^R}\lambda_{X^R}^{X} \otimes \sigma_{X_{S}}^{X^S}\rho^{X^{RS}}_{X^S}\lambda^{X}_{X^{RS}} \otimes \sigma_{X_{T}}^{X^{T}}\rho^X_{X^{T}}.
\intertext{In the last step, we use coassociativity of $\Delta^{\b}$ in the first tensor factor. The clockwise computation yields}
& \ 
\left[\id_{X_R}\otimes\bigl(\sigma_{(X_{ST})_S}^{(X_{ST})^S}\lambda^{X_{ST}}_{(X_{ST})^S} \otimes \sigma_{(X_{ST})_T}^{(X_{ST})^T}\rho^{X_{ST}}_{(X_{ST})^T}\bigr)\right]\circ\bigl(\sigma_{X_{R}}^{X^{R}}\lambda^X_{X^{R}} \otimes \sigma_{X_{ST}}^{X^{ST}}\rho^X_{X^{ST}}\bigr) \\
=& \ \sigma_{X_{R}}^{X^R}\lambda_{X^R}^{X} \otimes \sigma_{X_{S}}^{X^S}\lambda^{X^{ST}}_{X^S}\rho^{X}_{X^{ST}} \otimes \sigma_{X_{T}}^{X^{T}}\rho^X_{X^{T}}.
\end{align*}
We'd be done if the comonoid structure for $\b$ guaranteed that 
\begin{gather}
\label{eq:need-cocommutative}
	\rho^{X^{RS}}_{X^S}\lambda^{X}_{X^{RS}} = \lambda^{X^{ST}}_{X^S}\rho^{X}_{X^{ST}} \quad(\forall\,R\sqcup S\sqcup T=I).
\end{gather}
The cocommutative hypothesis on $\b$ provides that $\lambda^I_U = \rho^I_U$ for all $U\subseteq I$; this clearly suffices. (For then, each side of \eqref{eq:need-cocommutative} equals $\rho^X_{X^S}$.)

Conclude that $(\b\scirc \p, \Delta)$ is a comonoid.
\end{proof}

\subsubsection{A remark on the necessity of the cocommutative hypothesis}
In the preceding proof of coassociativity, it was sufficient to assume $\b$ is cocommutative. We now show this assumption is necessary (barring any additional assumptions on $\p$).

Consider the situation of Figure \ref{fig:bop-coassociativity}, with $X=\{X_1, \ldots, X_7\}$. Note $X^{RS}$ may be written as $X\setminus \{X_3\}$, or simply $\hat3$, and likewise $X^S$ is written as $\widehat{134}$. Then \eqref{eq:need-cocommutative} is rewritten as $\rho^{\hat3}_{\widehat{134}}\lambda^X_{\hat3} = \lambda^{\hat1}_{\widehat{134}}\rho^X_{\hat1}$. Both of these maps factor through the space $\b\bigl[\widehat{13}\bigr]$, as illustrated in Figure \ref{fig:bop-coassociativity-fail}. (The indicated maps commute since $\Delta^{\b}$ is coassociative.)
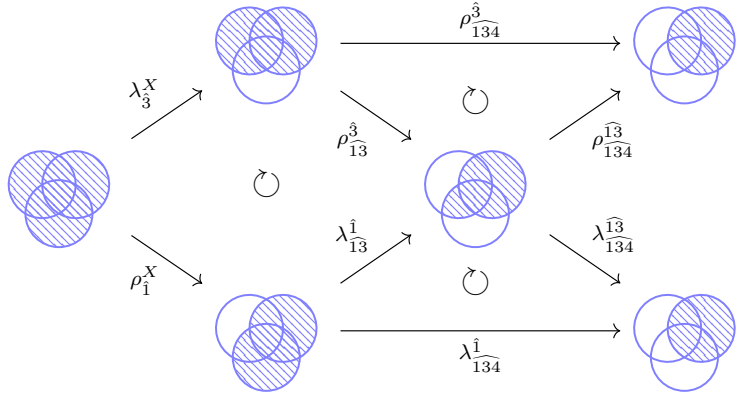
\begin{figure}[!hb]
\centering

\tikzset{every picture/.style={line width=0.75pt,x=0.75pt,y=0.75pt,yscale=-1,xscale=1,scale=.65},} 

\def\figone{%
\begin{tikzpicture}
\draw[color=blue!50!white,pattern=north west lines, pattern color=blue!50!white,rotate=-30,shift={(15,0)}] (0,0) circle (26);
\draw[color=blue!50!white,pattern=north west lines, pattern color=blue!50!white,rotate=-150,shift={(15,0)}] (0,0) circle (26);
\draw[color=blue!50!white,pattern=north west lines, pattern color=blue!50!white,rotate=90,shift={(15,0)}] (0,0) circle (26);

\draw[color=black,rotate=-150,shift={(54,0)}] node at (0,0) {};
\draw[color=black,rotate=-30,shift={(56,1)}] node at (0,0) {};
\draw[color=black,rotate=90,shift={(53,0)}] node at (0,0) {};
\end{tikzpicture}
}

\def\figtwo{%
\begin{tikzpicture}
\draw[color=blue!50!white,pattern=north west lines, pattern color=blue!50!white,rotate=-30,shift={(15,0)}] (0,0) circle (26);
\draw[color=blue!50!white,pattern=north west lines, pattern color=blue!50!white,rotate=-150,shift={(15,0)}] (0,0) circle (26);
\draw[color=blue!50!white,rotate=90,shift={(15,0)}] (0,0) circle (26);

\draw[color=black,rotate=-30,shift={(57,2)}] node at (0,0) {};
\draw[color=black,rotate=90,shift={(53,0)}] node at (0,0) {};
\end{tikzpicture}
}

\def\figthree{%
\begin{tikzpicture}
\draw[color=blue!50!white,pattern=north west lines, pattern color=blue!50!white,rotate=-30,shift={(15,0)}] (0,0) circle (26);
\draw[color=blue!50!white,rotate=-150,shift={(15,0)}] (0,0) circle (26);
\draw[color=blue!50!white,pattern=north west lines, pattern color=blue!50!white,rotate=90,shift={(15,0)}] (0,0) circle (26);

\draw[color=black,rotate=-30,shift={(57,2)}] node at (0,0) {};
\draw[color=black,rotate=90,shift={(53,0)}] node at (0,0) {};
\end{tikzpicture}
}

\def\figfour{%
\begin{tikzpicture}
\draw[color=blue!50!white,pattern=north west lines, pattern color=blue!50!white,rotate=-30,shift={(15,0)}] (0,0) circle (26);
\def\CR{
(-13,-8) circle (26)}
\def\CT{
(0,14.51666) circle (26)}
\begin{scope}
    \clip \CR;
    \fill[pattern=north west lines, pattern color=blue!50!white,] \CT;  
\end{scope}     
\draw[color=blue!50!white,rotate=-150,shift={(15,0)}] (0,0) circle (26);
\draw[color=blue!50!white,rotate=90,shift={(15,0)}] (0,0) circle (26);

\draw[color=black,rotate=-30,shift={(57,2)}] node at (0,0) {};
\draw[color=black,rotate=90,shift={(53,0)}] node at (0,0) {};
\end{tikzpicture}
}

\def\figfive{%
\begin{tikzpicture}
\draw[color=blue!50!white,pattern=north west lines, pattern color=blue!50!white,rotate=-30,shift={(15,0)}] (0,0) circle (26);
\draw[color=blue!50!white,rotate=-150,shift={(15,0)}] (0,0) circle (26);
\draw[color=blue!50!white,rotate=90,shift={(15,0)}] (0,0) circle (26);

\draw[color=black,rotate=-30,shift={(57,2)}] node at (0,0) {};
\draw[color=black,rotate=90,shift={(53,0)}] node at (0,0) {};
\end{tikzpicture}
}

\def\figsix{\figfive}

\begin{tikzcd}[ampersand replacement=\&,
cells={inner sep=+0pt, outer sep=1.0*width("$\;$")},
row sep=0em,column sep=3em,
]
    \& \figtwo \arrow[ddr, "\rho^{\hat3}_{\widehat{13}}"'] \arrow[rr, "\rho^{\hat3}_{\widehat{134}}"] \&\& \figfive \\[-2.5ex]
    \&\& {\hbox{\Large$\circlearrowright$}\ } \\[2.5ex]
  \figone \arrow[uur, "\lambda^X_{\hat3}"] \arrow[ddr, "\rho^X_{\hat1}"'] \& {\hbox{\Large$\circlearrowright$}\ } \& \figfour \arrow[uur, "\rho^{\widehat{13}}_{\widehat{134}}"'] \arrow[ddr, "\lambda^{\widehat{13}}_{\widehat{134}}"]  \\[.5ex]
    \&\& {\hbox{\Large$\circlearrowright$}\ } \\[-.5ex]
    \& \figthree \arrow[uur, "\lambda^{\hat1}_{\widehat{13}}"] \arrow[rr, "\lambda^{\hat1}_{\widehat{134}}"'] \&\& \figsix
\end{tikzcd}
\caption{A key step in the proof of coassociativity for $\b\scirc\p$, appealing where possible to the coassociativity axioms for $\b$.}
\label{fig:bop-coassociativity-fail}
\end{figure}
So we may compare 
\[
	\rho^{\hat3}_{\widehat{134}}\lambda^X_{\hat3} = \rho^{\widehat{13}}_{\widehat{134}}(\rho^{\hat3}_{\widehat{13}}\lambda^X_{\hat3}) 
	\qand
	\lambda^{\hat1}_{\widehat{134}}\rho^X_{\hat1} = \lambda^{\widehat{13}}_{\widehat{134}}(\lambda^{\hat1}_{\widehat{13}}\rho^X_{\hat1}) = \lambda^{\widehat{13}}_{\widehat{134}}(\rho^{\hat3}_{\widehat{13}}\lambda^X_{\hat3}),
\]
which must be equal for all $X\vdash I$, and $R\sqcup S\sqcup T=I$. Since $\lambda$ and $\rho$ are both surjective, {\it cf.} \cite{aguiar2010monoidal}, Cor.~8.38(i), we are asking that $\rho^{\widehat{13}}_{\widehat{134}}$ equals $\lambda^{\widehat{13}}_{\widehat{134}}$ on all of $\b\bigl[\widehat{13}\bigr]$, {\it i.e.}, that $\Delta^{\b}$ is cocommutative.

\subsection{Hopf monoid structure on $\b\scirc \p$}
\label{sec:bop-bimonoids}
We turn to the compatibility axioms.

\begin{theorem}\label{th:bop}
Let $\b$ be a connected cocommutative linearized bimonoid and $\p$ be a positive comonoid. Then the maps $\mu$ in \eqref{eq:bop-product} and $\Delta$ in \eqref{eq:bop-coproduct} afford $\b\scirc \p$ the structure of bimonoid in $(\Spec,\sdot)$.
\end{theorem}

We write \demph{$\tee{\b}{\p}$} for this construction and note that $\cfT(\mhyphen) = \tee{\bfL}{\mhyphen}$ and $\cfS(\mhyphen) = \tee{\bfE}{\mhyphen}$. 

\begin{proof}

\noindent We check compatibility of $\Delta$ and $\mu$.
Recall the shorthand, {\it e.g.}, $b\vert^{\downarrow}_{X_U}$ for $\sigma_{X_U}^{X^U}\rho^X_{X^U}(b)$. We also indicate $\mu^\b(b\otimes d)$ just by $bd$. Given $S\sqcup T=I$, define $S', T'$ and $A, B, C, D$ as in \eqref{eq:ABCD}.

Starting from simple tensors $b\otimes p \in \b[X]\otimes \p(X)\subseteq (\b\scirc\p)[S]$ and $d\otimes q \in \b[Y]\otimes \p(Y)\subseteq (\b\scirc\p)[T]$, we head clockwise and counterclockwise in \eqref{eq:compatible-delta}.
We get 
\begin{align}
\notag
b\otimes p \otimes d\otimes {}&q : 
\\
\notag
\xrightarrow{\Delta\otimes\Delta} \ \  &	
\biggl[\sum_{(p)} b\vert^{\downarrow}_{X_A}\otimes p_{(X_A)} \otimes b\vert^{\downarrow}_{X_B}\otimes p_{(X_B)}\biggr]
\otimes
\biggl[\sum_{(q)} d\vert^{\downarrow}_{Y_C}\otimes q_{(Y_C)} \otimes d\vert^{\downarrow}_{Y_D}\otimes q_{(Y_D)}\biggr] 
\\
\notag
\xrightarrow{\id\otimes\tw\otimes\id} \ \ & 
\sum_{(p,q)} \biggl[b\vert^{\downarrow}_{X_A}\otimes p_{(X_A)} \otimes d\vert^{\downarrow}_{Y_C}\otimes q_{(Y_C)}\biggr] \otimes \biggl[b\vert^{\downarrow}_{X_B}\otimes p_{(X_B)}\otimes d\vert^{\downarrow}_{Y_D}\otimes q_{(Y_D)}\biggr]
\\
\label{eq:bop-cw}
\xrightarrow{\mu\otimes\mu} \ \ & 
\sum_{(p,q)} \biggl[\bigl(b\vert^{\downarrow}_{X_A} d\vert^{\downarrow}_{Y_C}\bigr) \otimes \bigl(p_{(X_A)}\otimes q_{(Y_C)}\bigr)\biggr] \otimes \biggl[\bigl(b\vert^{\downarrow}_{X_B}  d\vert^{\downarrow}_{Y_D}\bigr)\otimes\bigl(p_{(X_B)}\otimes q_{(Y_D)}\bigr)\biggr]
\intertext{and (writing $XY$ for the partition $X\sqcup Y$ of $I$),}
\notag
b\otimes p \otimes d\otimes {} & q : \\
\notag
\xrightarrow{\,\,\mu\,\,} \quad&	bd\otimes(p\otimes q) 
\\
\label{eq:bop-ccw}
\xrightarrow{\,\,\Delta\,\,} \quad& 
\sum_{(p)(q)}
\biggl[(bd)\vert^{\downarrow}_{XY_{S'}} \otimes 
\bigl(p_{(X_{S'})}\otimes q_{(Y_{S'})}\bigr)\biggr] \otimes
\biggl[(bd)\vert^{\downarrow}_{XY_{T'}} \otimes \bigl(p_{(X_{T'})}\otimes q_{(Y_{T'})}\bigr)\biggr],
\end{align}
respectively. We now reconcile \eqref{eq:bop-cw} and \eqref{eq:bop-ccw}. 

\medskip\noindent\emph{Reconciling the factors in $\p$.} This is straightforward. Simply observe that: 
\[
	X_{S'} = X_A, \hbox{ and } Y_{S'} = Y_C; 
	\qand 
	X_{T'} = X_B, \hbox{ and } Y_{T'} = Y_D.
\]

\medskip\noindent\emph{Reconciling the factors in $\b$.} We focus on the left tensor-factors, decompressing notation a bit. We have
\begin{align*}
(bd)\vert^{\downarrow}_{XY_{S'}} &= \sigma_{XY_{S'}}^{XY^{S'}}\rho_{XY^{S'}}^{XY} \circ \mu_{X,Y}(b\otimes d) = \sigma_{XY_{S'}}^{XY^{S'}} \bigl(\mu_{X^{S'},Y^{S'}} \circ \rho_{X^{S'}}^{X} \otimes \rho_{Y^{S'}}^{Y}\bigr) (b\otimes d)
\intertext{and}
\bigl(b\vert^{\downarrow}_{X_A} d\vert^{\downarrow}_{Y_C}\bigr)
&= \mu_{X_A,Y_C} \circ \bigl(\sigma_{X_A}^{X^A}\rho_{X^A}^X \otimes \sigma_{Y_C}^{Y^C} \rho_{Y^C}^{Y}\bigr)(b\otimes d) 
\\
&= \bigl(\mu_{X_A,Y_C} \sigma_{X_A}^{X^A} \otimes \sigma_{Y_C}^{Y^C}\bigr) \circ \rho_{X^A}^X \otimes \rho_{Y^C}^{Y}(b\otimes d)
\\
&= \bigl(\sigma_{XY_{AC}}^{XY^{AC}}\mu_{X^A,Y^C}\bigr) \circ \rho_{X^A}^X \otimes \rho_{Y^C}^{Y}(b\otimes d)
\\
&= \sigma_{XY_{S'}}^{XY^{S'}}\bigl(\mu_{X^{S'},Y^{S'}} \circ \rho_{X^{S'}}^X \otimes \rho_{Y^{S'}}^{Y}\bigr)(b\otimes d).
\end{align*}
In the first calculation, we appeal to the coherence property of $\rho$ and $\mu$; in the second, to the fact that $\mu$ is a morphism in $(\Spec,\sdot)$. 

Reconciling the right tensor-factors proceeds similarly, and so the defining digram commutes.
Conclude that $\tee{\b}{\p}$ is a bimonoid in $(\Spec, \sdot)$. 
\end{proof}


\begin{corollary}
If $\b$ is a finite cocommutative linearized Hopf monoid in $(\Spec, \sdot)$ and $\p$ is a positive comonoid, then $\tee{\b}{\p}$ is a Hopf monoid in $(\Spec,\sdot)$.
\end{corollary}

\begin{proof}
We have verified that $\tee{\b}{\p}$ is a bimonoid. In lieu of verifying the antipode axioms, we appeal to \cite[Prop. 8.10]{aguiar2010monoidal}: a bimonoid $\h$ is a Hopf monoid if and only if $\h[\emptyset]$ is a Hopf algebra. Observe that $(\b\scirc\p)[\emptyset]\cong\b[\emptyset]=\Bbbk$ to complete the proof.
\end{proof}

\begin{example}
\label{sec:basic-examples}
A description of the bimonoid structure for our (new) running example $\tee{\bfG}{\bfL_+}$ is given in Figures \ref{fig:GoL-product} and \ref{fig:GoL-coproduct}. We highlight here an old example. (See also Example \ref{ex:endofunctions}.)

\smallskip
In \cite[\S13.4.2]{aguiar2010monoidal}, the authors give linearized Hopf monoid structures $\mathbf{O}_1$ and $\mathbf{P}_1$ on the species of preposets and posets, respectively. In Section 13.1.6, {\it ibid.}, they remark that $\mathbf{O}_1 = \mathbf{P}_1\scirc \bfE_+$ as species. It is a simple matter to check that $\mathbf{O}_1 \cong \tee{\mathbf{P}_1}{\bfE_+}$. We leave a careful description of the isomorphism to the interested reader.
\end{example}

\section{Structure of $\tee{\b}{\p}$}

See \cite[Ch. 15]{aguiar2010monoidal} for definitions of functors $\K, \Kbar, \Kvee$ from species to graded vector spaces. (Being \emph{bilax}, they take connected Hopf monoids to Hopf algebras.)

\subsection{Basic results}

In \cite[App.~B]{aguiar2010monoidal}, we learn that $(\Spec_+, \scirc)$ is a monoidal category---meaning, among other things, that $(\b\scirc\p)\scirc\q \cong \b\scirc(\p\scirc\q)$ for positive species $\p{,}\q$. One might ask \emph{how do the constructions $\tee{\b\scirc\p}{\q}$ and $\tee{\b}{\p\scirc\q}$ compare?} The reader will be happy with the answer.


\begin{proposition} Let $\b$ be a cocommutative linearized Hopf monoid, and let $\p{,}\q$ be positive comonoids. If $\p$ is a cocommutative linearized comonoid, then $\tee{\tee{\b}{\p}}{\q} \cong \tee{\b}{\p\scirc\q}$.
\end{proposition}

We establish some notation before beginning the proof.
In the running notation, a simple tensor in $(\b\scirc\p)\scirc\q$ takes the form
\begin{gather}
\label{eq:(bp)q}	
(bp)_X\otimes q_{(X)}
	= (b_{\mathcal X}\otimes p_{(\mathcal X)})\otimes 
	q_{(X)}
	= (b_{\mathcal X}\otimes_{\mathcal A_i\in \mathcal X} p_{\mathcal A_i})\otimes_{ij} q_{A_{ij}}.
\end{gather}
Here $\mathcal X \vdash X \vdash I$. Let us write $X = \{A_1, A_2, \ldots\}$ and $\mathcal X = \{\mathcal A_1, \ldots, \}$, with each $\mathcal A_i$ a collection of some blocks $A_{ij}\in X$ (after relabelling said blocks).

By contrast, a simple tensor in $\b\scirc(\p\scirc\q)$ takes the form
\begin{gather}
\label{eq:b(pq)}
	b_{Y}\otimes (pq)_{(Y)} 
	= b_{Y}\otimes_{B_i \in Y} (p_{\mathcal B_i}\otimes q_{(\mathcal B_i)})
	= b_{Y}\otimes_{B_i\in Y} (p_{\mathcal B_i} \otimes_{j} q_{B_{ij}}).
\end{gather}
Here, we write $Y = \{B_1, \dots\}\vdash I$ and $\mathcal B_i \vdash B_i$, with $\mathcal B_i = \{B_{i1}, B_{i2},\ldots\}$.

\begin{proof}

The species isomorphism $\varphi:(\b\scirc\p)\scirc\q \to \b\scirc(\p\scirc\q)$ is built from the correspondence between the partitions $\mathcal A_i$ in $\mathcal X$ and the blocks $B_i$ in $Y$. Specifically, $\varphi$ uses the natural bijection $\b[\mathcal X] \xrightarrow{\tau} \b[Y]$ and shuffles the $\p$ and $\q$ components as appropriate. (Identify $\mathcal A_i$ with $\mathcal B_i$ and $A_{ij}$ with $B_{ij}$.) We must check that it correctly transports the monoid and comonoid structures. The former check is straightforward and omitted.

For the comonoid check, we make essential use of the fact that $\b, \p$, and $\b\scirc \p$ have their coproduct defined in terms of restrictions. 
Consider an element $(bp)_X \in (\b\scirc\p)[X]$. Given $\mathcal X\vdash X\vdash I$ and $\mathcal A_i \vdash B_i$ as above, and $S\subseteq I$, note that the compact expression $(bp)_X\vert_{X^S}$ takes the expanded form $b\vert_{\mathcal X^{X^S}} \otimes_i p_i\vert_{\mathcal A_i^{S_i}}$. (Here, $b\in \b[\mathcal X]$,  $p_i\in\p[\mathcal A_i]$, and $S_i = B_i \cap S$.)

So, writing the coproduct on $q\in\q[A_{ij}]$ as $\sum_{(q_{A_{ij}})} q_{A_{ij}'}\otimes q_{A_{ij}''}$, we have in $(\b\scirc\p)\scirc\q$:
\[
\Delta_{S,T} \bigl((bp)_X\otimes q_{(X)}\bigr)
	= \sum_{ij}\sum_{(q_{A_{ij}})}\left[(bp)\vert^\downarrow_{X_{X_S}}\otimes_{ij} q_{A_{ij}'}\right]
	\otimes
	\left[(bp)^\downarrow_{X_{X_T}} \otimes_{ij} q_{A_{ij}''}\right],
\]
with $(bp)_X\vert^\downarrow_{X_S} = (b\vert_{\mathcal X^{X^S}})^\downarrow_{\mathcal X_{X_S}} \otimes_i (p_i\vert_{{A_i}^{S_i}})^\downarrow_{{\mathcal A_i}_{S_i}}$, as mentioned above. Likewise for $(bp)_X\vert^\downarrow_{X_T}$.

As we apply $\varphi^{\otimes2}$ to the term
\begin{gather*}
\label{eq:bpq-proof-step}
\left[(b\vert_{\mathcal X^{X^S}})^\downarrow_{\mathcal X_{X_S}} \otimes_i (p_i\vert_{{A_i}^{S_i}})^\downarrow_{{\mathcal A_i}_{S_i}} \otimes_{ij}q_{A_{ij}'}\right]
\otimes
\left[(b\vert_{\mathcal X^{X^T}})^\downarrow_{\mathcal X_{X_T}} \otimes_i (p_i\vert_{{A_i}^{T_i}})^\downarrow_{{\mathcal A_i}_{T_i}} \otimes_{ij}q_{A_{ij}''}\right],
\end{gather*}
we note that identities in the spirit of \eqref{eq:restriction-is-a-morphism} come into play; for example, one has
\[
    \tau\left(\bigl(b\vert_{{\mathcal X}^{X^S}}\bigr)^{\downarrow}_{\mathcal X_{X_S}}\right) 
    = 
    \left(\tau(b)\vert_{Y^S}\right)^{\downarrow}_{Y_S}.
\]
The final result is indeed yield the corresponding term of $\Delta_{S,T}\circ\varphi\bigl((bp)_X\otimes q_{(X)}\bigr)$.
\end{proof}

\begin{example}
\label{ex:endofunctions}
Consider the species of endofunctions \cite[Sec. 1.4]{bergeron1998combinatorial}. This is naturally $\bfE\scirc(\cyc\scirc\a)$, where $\cyc$ and $\a$ are the positive species of cycles and rooted trees, respectively. 
(View $\cyc\scirc\a$ as the species of 
\emph{connected} endofunctions.) Alternatively, it is $(\bfE\scirc\cyc)\scirc\a = \bij\scirc\a$, where $\bij$ is the species of bijections \cite[Ex. 11.16]{aguiar2010monoidal}. 

The preceding proposition tells us $\tee{\bfE}{\cyc\scirc\a} \cong \tee{\bij}{\a}$, whenever $\bij=\cfS(\cyc)$ holds a cocommutative linearized comonoid structure on $\cyc$. This gives a new realization of the Hopf algebra of endofunctions from \cite[Sec. 2]{hivert2008commutative}, as $\K(\tee{\bfE}{\cyc\scirc\a})$. (Both $\cyc$ and $\a$ are given the trivial coproduct there, so $\cyc$ is linearized cocommutative.)

In \cite[Ex. 2.10]{marberg2015linearization} one finds another structure on $\bij$, starting from a nontrivial linearized cocommutative comonoid structure on $\cyc$. This gives the new isomorphism $\K(\tee{\bfE}{\cyc\scirc\a}) \cong \K(\tee{\bij}{\a})$ added interest. 
\end{example}

\begin{proposition} 
\label{th:bop-to-doq}
Suppose $\tau:\b \to \d$ is the linearization of a morphism of species with restrictions and moreover is a morphism of cocommutative linearized bimonoids. Suppose $\theta: \p \to \q$ is a morphism of positive comonoids. Then the mapping $f_{\tau,\theta}:\tee{\b}{\p} \to \tee{\d}{\q}$ given on simple tensors by
\[
b_X \otimes 
\bigotimes_{X_i \in X} p_{X_i} \ \mapsto \ \tau(b_X) \otimes 
\bigotimes_{X_i \in X} \theta(p_{X_i})
\]
is a morphism of Hopf monoids.
\end{proposition}

The \demph{abelianization map} $\pi_\p:\cfT(\p) \to \cfS(\p)$ of \cite[Sec. 11.6.2]{aguiar2010monoidal} is the special case of this result when $\tau:\bfL \to \bfE$ forgets the linear order and $\theta:\p\to\p$ is the identity map. This gives, {\it e.g.}, the well-known Hopf algebra maps $\nsym \to \sym$ and $\ncqsym^* \to \ncsym^*$ under the Fock functors $\Kbar$ and $\Kvee$, with $\p=\bfE_+$.

Another example comes from $\tee{\bfG}{\bfL_+}$ mapping to various other species, {\it e.g.}, $\tee{\bfG}{\bfL_+} \to \tee{\bfE}{\bfE_+}$. Here $\tau: \bfG \to \bfE$ sends a graph in $\rfG[I]$ to the set partition induced of $I$ induced by its connected components; and $\theta$ is the map called ``$\tau$'' in the previous paragraph. We give a final example after the proof.

\begin{proof}
We show that $f_{\tau,\theta}$ respects coproducts. 
\begin{align*}
\Delta_{S,T}(f_{\tau,\theta})(b_X\otimes p_{(X)}) 
  &= \Delta_{S,T}(\tau(b_X)\otimes \theta(p_{(X)})) \\
  &= \sum_{(\theta(p_{(X)}))} \tau(b_X)\vert^\downarrow_{X_S} \otimes {\theta(p_{(X)})}_S \otimes
  \tau(b_X)\vert^\downarrow_{X_T} \otimes {\theta(p_{(X)})}_T
\\
  &= \sum_{(p_{(X)})} \tau(b_X)\vert^\downarrow_{X_S} \otimes \theta(p_{(X_S)}) \otimes
  \tau(b_X)\vert^\downarrow_{X_T} \otimes \theta(p_{(X_T)})
\\  
  &= \sum_{(p_{(X)})} \tau(b_X\vert^\downarrow_{X_S})\otimes \theta(p_{(X_S)}) \otimes
  \tau(b_X\vert^\downarrow_{X_T}) \otimes \theta(p_{(X_T)})
\\  
  &= f_{\tau,\theta}^{{}^{\,\otimes2}}\biggl(\sum_{(p_{(X)})} {b_X}\vert^\downarrow_{X_S}\otimes p_{(X_S)} \otimes
  {b_X}\vert^\downarrow_{X_T} \otimes p_{(X_T)}\biggr) \\
  &=f_{\tau,\theta}^{{}^{\,\otimes2}}\Delta_{S,T}(b_X\otimes p_{(X)}), 
\end{align*}
as needed. The key step above relies on $\tau$ intertwining with restrictions---it is not enough for $\tau$ to be a comonoid map because, recall, $a\mapsto a\vert^\downarrow_{X_S} \otimes a\vert^\downarrow_{X_T}$ is not a coproduct in $\b$. 

The straightforward proof that $f_{\tau,\theta}$ respects products is omitted.\end{proof}

\begin{example}
Continuing the discussion from Example \ref{ex:basic examples}, let $\mathbf{Q}_1$ denote the Hopf monoid of equivalence relations from \cite[\S13.4.3]{aguiar2010monoidal}. The authors explain that $\bfPi\cong \mathbf{Q}_1$ and argue that $\mathbf{Q}_1$ embeds in $\mathbf{O}_1$. The latter is now immediate from the former (viewing $\bfPi=\bfE\scirc\bfE_+$ and $\mathbf{O}_1 = \mathbf{P}_1\scirc\bfE_+$) using the map $f_{\tau,\theta}$ from the proposition: let $\tau$ be the inclusion $\alpha:\bfE \to \mathbf{P}_1$ taking $\ast_I\in\rfE[I]$ to the antichain in $\mathrm{P}[I]$; and let $\theta:\bfE_+ \to \bfE_+$ be the identity map.

We next describe bimonoid morphisms from $\mathbf{O}_1$ to $\mathbf{P}_1$. They will factor through the Hopf monoid $\tee{\mathbf{P}_1}{\mathbf{P}_{1,+}}$. Note that $\bfE$ embeds in $\mathbf{P}_1$ as a comonoid in another way (aside from $\alpha$ above):
\[
    \lambda:\bfE[I] \into \mathbf{P}[I]
    \quad\text{satisfying}\quad
    \ast_I \mapsto \frac{1}{|I|!}\sum_{\ell\in\rfL[I]} \ell.
\]
This is not a bimonoid morphism (nor a morphism of species with restrictions), but it does give us now two new instances of the proposition---specifically, $f_{\id,\theta}: \tee{\mathbf{P}_1}{\bfE_+} \to \tee{\mathbf{P}_1}{\mathbf{P}_{1,+}}$ with $\theta\in\{\alpha,\lambda\}$. 

Finally, we describe a forgetful map $\varphi:\tee{\mathbf{P}_1}{\mathbf{P}_{1,+}} \to \mathbf{P}_1$. Given a set partition $X\vdash I$, a poset $p_X\in\mathrm{P}[X]$, and posets $\bigl\{p_A\in\mathrm{P}[A]\bigr\}_{A\in X}$, we let $\varphi\bigl(p_X\otimes p_{(X)}\bigr)$ be the obvious poset in $\mathrm{P}[I]$ induced by this data. It is straightforward to check that $\varphi$ is a morphism of bimonoids. Composing $\varphi$ with the maps from the preceding paragraph, we get two bimonoid surjections $\mathbf{O}_1 \onto \mathbf{P}_1$ splitting the inclusion $\mathbf{P}_1 \into \mathbf{O}_1$. 
\end{example}

\subsection{Towards a universal property}
\label{sec:towards-a-up}

There does not seem to be a simple interpretation of the bifunctor $\tee{\mhyphen}{\mhyphen}$ as producing universal objects in $\bimon(\Spec,\sdot)$. Given a morphism $\zeta:\p\to\h_+$ of positive comonoids, we would like a morphism $\hat{\zeta}:\b\scirc\p\to\h$ of bimonoids. This amounts to needing a species morphism $\chi:\b\scirc\h_+ \to \h$ making the following diagrams commute
\begin{gather}
\label{eq:pre-operad}
\raisebox{.7\height}{\begin{tikzcd}[ampersand replacement=\&]
    (\b\scirc\h_+)\sdot(\b\scirc\h_+) \arrow[r, "\chi^{\sdot2}"]\arrow[d, "\mu^{(\b\scirc\h_+)}"'] \&  \h \sdot \h \arrow[d, "\mu^{\h}"]\\
    (\b\scirc\h_+) \arrow[r, "\chi"'] \&  \h
\end{tikzcd}}
\qquad
\raisebox{.7\height}{\begin{tikzcd}[ampersand replacement=\&]
    (\b\scirc\h_+) \arrow[r, "\chi^{\sdot2}"]\arrow[d, "\Delta^{(\b\scirc\h_+)}"'] \&  \h \arrow[d, "\Delta^{\h}"]\\
    (\b\scirc\h_+)\sdot(\b\scirc\h_+) \arrow[r, "\chi"'] \&  \h \sdot \h
\end{tikzcd}} \,.
\end{gather}

If $\b=\bfL$, this exists for any bimonoid $\h$; if $\b=\bfE$, this exists for any commutative bimonoid $\h$. Concretely for $\bfL$: if $X=\{B_1, \ldots, B_k\}\vdash I$, then $\chi$ takes $l_X\otimes_i h_{B_i}$ to the iterated product in $\h$ in the order prescribed by $l$. Checking \eqref{eq:pre-operad} is easy in these cases. We leave further analysis to future work, but highlight two parallels with $\cfT(\p)$.

\begin{proposition} 
\label{th:p-embeds}
Suppose $\dim \b[I] \geq 1$ for singleton sets $I$. Then the positive comonoid $\p$ embeds as a subcomonoid of $\tee{\b}{\p}_+$. \qed
\end{proposition}

\begin{proposition} 
\label{th:a-embeds}
Suppose $\dim \p[I] \geq 1$ for singleton sets $I$. Then the Hopf monoid $\b$ embeds as a Hopf submonoid of $\tee{\b}{\p}$. \qed
\end{proposition}

%
%

\section{The Hopf monoid $\rspec{\tee{\b\d}{\p\q}}$}

Motivated by a desire to better understand the Hopf algebra $\rqsym$ of $r$-quasi\-symmetric functions, the authors in \cite{lauve202xinterpolation} introduce a notion of interpolation between Hopf monoids in species. Let $\theta:\p\to\q$ be any morphism of positive comonoids. Composing with the embedding $\eta_{\q}:\q \to \cfT(\q)$, we get a bimonoid map $\hat\theta:\cfT(\p)\to\cfT(\q)$ granted by the universal property of $\cfT(\p)$. The main result of \cite{lauve202xinterpolation} may be summarized as follows. There is a one-parameter family of Hopf monoids $\bigl\{\rc\bigr\}_{r\geq1}$ and Hopf maps 
$\port^r_s:\rc \to \sc$ 
interpolating between $\cfT(\p)$ and $\cfS(\q)$ and factoring the Hopf map $\pi_\theta := \pi_\q\circ\hat{\theta}$. We extend this result to our construction $\tee{\mhyphen}{\mhyphen}$.

\begin{definition}
Given $r\in\ZZ_{\geq0}$ and species $\b,\p$ as in the running notation, let $\upspec{({\b}\scirc{\p})}$ and $\downspec{({\b}\scirc{\p})}$ be the species $\b\scirc \p_{\geq r}$ and $\b\scirc \p_{< r}$, respectively. If $\d,\q$ are additional species, with $\q$ positive, put
\[
	\rspec{\tee{\b,\d}{\p{,}\q}} := \upspec{({\b}\scirc{\p})} \sdot \downspec{({\d}\scirc{\q})}.
\]
(The special cases $\b=\bfL, \d=\bfE$ yield the species $\rc$ mentioned above.)
\end{definition}

\begin{theorem}\label{th:interpolating} 
Given connected cocommutative linearized bimonoids $\b,\d$ and two positive comonoids $\p{,}\q$. Suppose $\d$ is commutative, $\tau:\b\to\d$ is a morphism of species with restrictions and a bimonoid map, and $\theta:\p\to\q$ is a comonoid map. Then $\bigl\{\rspec{\tee{\b,\d}{\p{,}\q}}\bigr\}_{r\geq1}$ is a one-parameter family of Hopf monoids interpolating between $\tee{\b}{\p}$ and $\tee{\d}{\q}$. In particular, there are bimonoid morphisms 
\[
	\tee{\b}{\p} \xrightarrow{\port_r} \rspec{\tee{\b,\d}{\p{,}\q}}, 
	\ \ \ 
	\rspec{\tee{\b,\d}{\p{,}\q}} \xrightarrow{\port_s^r} \rspec[s]{\tee{\b,\d}{\p{,}\q}}, 
	\ \ \ 
	\rspec[s]{\tee{\b,\d}{\p{,}\q}} \xrightarrow{\port^s} \tee{\d}{\q}, 
\]
factoring a map $\hat{f}_{\tau,\theta}$ from $\tee{\b}{\p}$ to $\tee{\d}{\q}$.
These morphisms are surjective if $\hat{f}_{\tau,\theta}$ is.
\end{theorem}

The map $\hat{f}_{\tau,\theta}$ is related to the map from Proposition \ref{th:bop-to-doq}, but different. 

\begin{remark} 
A proof of this result cannot simply verify the interpolation axioms given in \cite[Sec. 3]{lauve202xinterpolation}. (Those axioms hold, but the complements $\upspec{(\b\scirc\p)}$ built there can be much larger than the species $\b\scirc\p_{\geq r}$ we consider here.)
\end{remark}

As in the main result of \cite{lauve202xinterpolation}, the proof begins by analyzing slightly larger Cauchy products. Put $\teebar{\h} := \tee{\b}{\p} \sdot \tee{\d}{\downspec{\q}}$, which is a bimonoid with ``coordinatewise'' product and coproduct, {\it cf.} \cite[Sec. 1.2.7]{aguiar2010monoidal}.
(In what follows, we use ``$\dototimes$'' to distinguish between factors from $\tee{\b}{\p}$ and $\tee{\d}{\downspec{\q}}$ in $\teebar{\h}$.) 

Lemma \ref{th:coideal} establishes $\rspec{\tee{\b,\d}{\p{,}\q}}$ as a (bimonoid) quotient of $\teebar{\h}$. 
Given a partition $X\vdash I$, call a block $B\in X$ \emph{large} (otherwise, \emph{small}) if $|B|\geq r$. Let $X = X'\sqcup X''$ be the decomposition of the partition $X$ into its large and small blocks, respectively.

\begin{lemma}\label{th:coideal} 
Let $\b=\Bbbk\rfB$, and let $\rspec{\!\mathcal I}$ be the subspace of $\teebar{\h}$ spanned by elements
\begin{gather*}
\label{eq:rI-spanners}
\mathscr{S} = \left\{ b\otimes p_{(X)}\dototimes d\otimes q_{(Y)} - b\vert_{X'}\otimes p_{(X')}
\dototimes
\tau(b\vert_{X''})d\otimes \theta(p_{(X'')})\otimes q_{(Y)}
 \ \mid \ b\in \rfB[X] \right\}
\end{gather*}
in $\teebar{\h}[I\sqcup J]$ where: $X\vdash I$ with $X''\neq \emptyset$; $Y\vdash J$; and $b \mapsto b\vert_{X'}\otimes b\vert_{X''}$ is the coproduct of $\b$. 
Then $\rspec{\!\mathcal I}$ is an ideal and the quotient $\teebar{\h}/\rspec{\!\mathcal I}$ is isomorphic to $\rspec{\tee{\b,\d}{\p{,}\q}}$ as species. Moreover, as $\d$ is commutative, $\rspec{\!\mathcal I}$ is also a coideal.
\end{lemma}

\begin{proof} In three parts.

\smallskip\noindent
\emph{Claim: $\rspec{\!\mathcal I}$ is an ideal.} 
Given any $h\in \teebar{\h}$ and any $x \in\Span{\mathscr{S}}$, we argue that $hx$ is again in $\Span{\mathscr{S}}$. It suffices to take $x\in\mathscr{S}$. The factors originating from $\p,\d,\q$ don't offer any complications, so we suppress them in the notation. That is, suppose below that $h \in \rfB[H]$ and $x=b\dototimes1 - b\vert_{X'}\dototimes\tau(b\vert_{X''})$ for some $b \in \rfB[X]$ with $X$ not small. Since $\b$ is a linearized bimonoid, we have from the coherence condition \eqref{eq:rescriction-coherence} that 
\[
	(h\cdot b\vert_{X'})\vert_{H'\sqcup X'}
	= h\vert_{H'}\cdot (b\vert_{X'})\vert_{X'} 
	= h\vert_{H'}\cdot b\vert_{X'}.
\]
Since $\b$ is also connected, we have $(b\vert_{X'})\vert_{X'\cap H''} = (b\vert_{X'})\vert_{\emptyset} = 1_{\b}$. 
We use these identities and the fact that $\tau$ is a bimonoid map, as follows:
\begin{align*}
hx&=hb\dototimes1 \ - \ hb\vert_{X'}\dototimes \tau(b\vert_{X''}) \\
&=hb\dototimes1 \ - \ h\vert_{H'}b\vert_{X'}\dototimes\tau(h\vert_{H''})\tau(b\vert_{X''}) \\
&\qquad\qquad \ \ + \ h\vert_{H'}b\vert_{X'}\dototimes\tau(h\vert_{H''})\tau(b\vert_{X''}) \ - \ hb\vert_{X'}\dototimes \tau(b\vert_{X''})
 \\
&=\Bigl(hb\dototimes1 - (hb)\vert_{(H\sqcup X)'} \dototimes \tau\left((hb)\vert_{(H\sqcup X)''}\right)\Bigr) \\
&\qquad \ \ - 
\Bigl( (hb\vert_{X'})\dototimes \tau(b\vert_{X''}) \ - \ 
(hb\vert_{X'})\vert_{(H\sqcup X')'} \dototimes \tau((hb\vert_{X'})\vert_{(H\sqcup X')''})\tau(b\vert_{X''})\Bigr).
\end{align*}
Finally, since $\b$ is a linearized monoid, all products above ({\it e.g.}, $hb\vert_{X'}$) represent basis elements in $\rfB[I]$ for appropriate $I$. Deduce that $hx$ belongs to $\Span{\mathscr{S}}$, as needed. A similar computation shows the same for $xh$.

\medskip\noindent
\emph{Claim: The quotient is isomorphic to $\rspec{\tee{\b,\d}{\p{,}\q}}$.}
Clearly any element of the quotient has a coset representative in $\rspec{\tee{\b,\d}{\p{,}\q}} \subseteq \teebar{\h}$. Concretely, 
\[
    b\otimes p_{(X)}\dototimes d\otimes q_{(Y)} \equiv b\vert_{X'}\otimes p_{(X')}\dototimes \tau(b\vert_{X''})d\otimes \theta(p_{(X'')})\otimes q_{(Y)} \mod \rspec{\!\mathcal I}. 
\]
Suppose $h \in \rspec{\tee{\b,\d}{\p{,}\q}}[I\sqcup J]$ satisfies $h\equiv 0 \mod \rspec{\!\mathcal I}$. We argue that $h=0$. 

\smallskip
(1) Pick any ordering of a basis for $\teebar{\h}[I\sqcup J]$ that sorts first by cardinality of $X$. (2) Extract a basis $\tilde{\mathscr{S}}$ for $\rspec{\!\mathcal I}$ from $\mathscr{S}$. (3) The leading term of any 
\[
    \tilde{x} = b\otimes p_{(X)}\dototimes d\otimes q_{(Y)} - b\vert_{X'}\otimes p_{(X')}\dototimes \tau(b\vert_{X''})d\otimes \theta(p_{(X'')})\otimes q_{(Y)}
\]
in $\mathscr{S}$ is $b\otimes p_{(X)}\dototimes d\otimes q_{(Y)}$ (with $X$ not large); and all leading terms in $\mathscr{S}$ are distinct. (4) As $h\in \rspec{\tee{\b,\d}{\p{,}\q}}$, its leading term corresponds to some partitions $(X',Y)$ with $X'$ large. (5) On the other hand, as $h = \sum_{\iota} \alpha_\iota \tilde{x}_\iota$ (for some $\tilde{x}_\iota \in \mathscr{S}$, $\alpha_\iota\in\Bbbk$), its leading term is the same as that of the first $\tilde{x}_\iota$ with $\alpha_\iota\neq0$; but that leading term comes from some partitions $(X,Y)$ with $X$ not large. Conclude all $\alpha_iota$ are zero.

\medskip\noindent
\emph{Claim: $\rspec{\!\mathcal I}$ is a coideal.} 
We must show that any $x$ from $\mathscr{S}$ has the feature that $\Delta(x) \in \teebar{\h}\sdot \rspec{\!\mathcal I} + \rspec{\!\mathcal I}\sdot \teebar{\h}$. A clever choice of adding zero (similar to what was done for the ideal check) gives the result. Again, factors originating from $\p,\d,\q$ don't offer any complications, so we suppress them in the notation. (Though we do make reference to $\p$, as in \eqref{eq:bop-coproduct}, when using the modified Sweedler notation for the coproduct.)

Given $x=b\dototimes1 - b\vert_{X'}\dototimes\tau(b\vert_{X''})$ as above, we have
\begin{gather*}
\Delta_{S,T}(x) = \sum_{(p_{(X)})} b\vert^\downarrow_{X_S} \dototimes 1 \otimes b\vert^\downarrow_{X_T}\dototimes1 \ - \! \sum_{\substack{(p_{(X')})\\(\theta(p_{(X'')}))}} b\vert^\downarrow_{X'_S} \dototimes \tau(b)\vert^\downarrow_{X''_S} \otimes b\vert^\downarrow_{X'_T}\dototimes \tau(b)\vert^\downarrow_{X''_T}
\end{gather*}
where, recall, $(b\vert_{X'})\vert_{{X'}^S} = b\vert_{{X'}^S}$, and we have again suppressed the bijections $\sigma^{X^S}_{X_S}$ in the notation. Recall also that the coproduct in $\p(X)$ is computed factor-by-factor, and $\theta$ is a comonoid map. So the sums $\sum_{(p_{(X)})}$ and $\sum_{\substack{(p_{(X')})\\(\theta(p_{(X'')}))}}$ agree as far as factors from $\p$ and $\q$ are concerned; we write them each as $\sum_{(p_{(X)})}$ in what follows. 
Continuing,
\begin{align*}
&= \sum_{(p_{(X)})} b\vert^\downarrow_{X_S} \dototimes 1 \otimes b\vert^\downarrow_{X_T}\dototimes1 \ - \sum_{(p_{(X)})} b\vert^\downarrow_{X_S} \dototimes 1 \otimes b\vert^\downarrow_{X'_T}\dototimes \tau(b)\vert^\downarrow_{X''_T} \\
& \quad \ + \ \sum_{(p_{(X)})} b\vert^\downarrow_{X_S} \dototimes 1 \otimes b\vert^\downarrow_{X'_T}\dototimes \tau(b)\vert^\downarrow_{X''_T} \ - \sum_{(p_{(X)})} b\vert^\downarrow_{X'_S} \dototimes \tau(b)\vert^\downarrow_{X''_S} \otimes b\vert^\downarrow_{X''_T}\dototimes \tau(b)\vert^\downarrow_{X''_T}
\\
&= \sum_{(p_{(X)})} \bigl(b\vert^\downarrow_{X_S} \dototimes 1\bigr) \otimes \left( b\vert^\downarrow_{X_T}\dototimes1 \ - \ b\vert^\downarrow_{X'_T}\dototimes \tau(b)\vert^\downarrow_{X''_T}\right) \\
& \quad \ + \ \sum_{(p_{(X)})} \left(b\vert^\downarrow_{X_S} \dototimes 1  \ - \ b\vert^\downarrow_{X'_S} \dototimes \tau(b)\vert^\downarrow_{X''_S} \right) \otimes \bigl(b\vert^\downarrow_{X'_T}\dototimes \tau(b)\vert^\downarrow_{X''_T}\bigr),
\end{align*}
which lives in the subspace $\teebar{\h}[S]\otimes \rspec{\!\mathcal I}[T] + \rspec{\!\mathcal I}[S]\otimes \teebar{\h}[T]$, as needed.
\end{proof}

After the lemma, it remains to define candidate factoring bimonoid maps and check that all relevant diagrams commute.

\begin{proof}[Proof of Theorem \ref{th:interpolating}]  
\underline{The map $\hat{f}_{\tau,\theta}$}.~Given a partition $X\vdash I$ with largest block-size $k$, let $X=X^1\sqcup X^2\sqcup \cdots \sqcup X^k$ be refinement of $X$ that groups blocks according to size and put $X^i \vdash S^i$.
We define $\hat{f}_{\tau,\theta} : \tee{\b}{\p} \to \tee{\d}{\q}$ via $(\tau\otimes \theta^{\otimes X})\circ(\mu^{\tee{\b}{\p}}_{S^1,\ldots,S^k}\circ\Delta^{\tee{\b}{\p}}_{S^1,\ldots,S^k})$. Or:
\[
	b\otimes p_{(X)} \ \mapsto \ \bigl(\tau\circ\mu^{\b}_{X^1,\ldots, X^k}\circ\Delta^{\b}_{X^1,\ldots,X^k}\bigr)(b) \otimes \theta(p_{(X)}) = \prod_{i}\tau(b)\vert_{X^i}\otimes \theta(p_{(X)}).
\]
Following the proof of Proposition \ref{th:bop-to-doq}, one verifies that this is a morphism of bimonoids (making heavy use of the fact that $\d$ is commutative).

\medskip
\underline{The maps $\port_r$ and $\port^r$}.~Given $b\otimes p_{(X)}$ as above, and integer $r\geq 2$, write $X=X'\sqcup X''$ for the decomposition of $X$ into large and small blocks, as in the lemma. Say $X'\vdash S$ and $X''\vdash T$.
Define $\port_r:\tee{\b}{\p} \to \rspec{\tee{\b,\d}{\p{,}\q}}$ as the composite $(\id\otimes \hat{f}_{\tau,\theta})\circ\Delta^{\tee{\b}{\p}}_{S,T}$. That is, 
\[
	b\otimes p_{(X)} \ \mapsto \ b\vert_{X'}\otimes p_{(X')}\dototimes \hat{f}_{\tau,\theta}\bigl(b\vert_{X''}\otimes p_{(X'')}\bigr).
\]
Since $\b$ is cocommutative, $\Delta$ is a bimonoid map, and so the composite with $\hat{f}_{\tau,\theta}$ is as well. 

Similarly, if $X\vdash I$ and $Y \vdash J$, define $\port^r:\rspec[r]{\tee{\b,\d}{\p{,}\q}} \to \tee{\d}{\q}$ as the composite $\mu^{\tee{\d}{\q}}_{I,J}\circ(\hat{f}_{\tau,\theta} \otimes \id)$. So: 
\[
	b\otimes p_{(X)}\dototimes d_{Y}\otimes q_{(Y)} \ \mapsto \ \prod_i\tau(b)\vert_{X^i}\cdot d_Y\otimes \theta\bigl(p_{(X)}\bigr)\otimes q_{(Y)}.
\]
This is a bimonoid map since $\d$ is commutative. 

\medskip
\underline{The map $\port_s^r$}.~Consider an element $b\otimes p_{(X)}\dototimes d\otimes q_{(Y)}$ in $\rspec{\tee{\b,\d}{\p{,}\q}}$ with $X\vdash I$ and $Y\vdash J$. (Recall $1\leq r < s <\infty$.) Write $X=X'\sqcup X''$ with large blocks defined now with respect to $s$, and put $X'\vdash S, X''\vdash T$. Take $\port^r_s$ to be the composite $(\id\otimes \mu^{\tee{\d}{\q}}_{T,J})\circ(\id\otimes \hat{f}_{\tau,\theta}\otimes \id)\circ(\Delta^{\tee{\b}{\p}}_{S,T}\otimes \id)$. Again, this is a bimonoid map since $\b$ is cocommutative and $\d$ is commutative. 

\medskip
It remains to check that the requisite identities from \cite{lauve202xinterpolation} hold. Most are straightforward; indeed, the identity $\port^r\circ\port_r = \hat{f}_{\tau,\theta}$ is trivial. We check that $\port_s = \port^r_s\circ\port_r$. 

Put $X=X'\sqcup X''\sqcup X'''$, where: a block $B$ lands in $X',X''$, or $X'''$ according to whether $|B|\geq s$, $s>|B|\geq r$, or $r>|B|$, respectively. Finally, define subsets $R,S,T$ via $X'\vdash R$, $X''\vdash S$, and $X'''\vdash T$. We have the simple computation
\begin{align*}
\port^r_s\circ\port_r &= (\id\otimes \mu^{\tee{\d}{\q}}_{S,T})\circ(\id\otimes \hat{f}_{\tau,\theta}\otimes \id)\circ(\Delta^{\tee{\b}{\p}}_{R,S}\otimes \id)\circ
(\id\otimes \hat{f}_{\tau,\theta})\circ\Delta^{\tee{\b}{\p}}_{RS,T} 
\\
&= (\id\otimes \mu^{\tee{\d}{\q}}_{S,T})\circ(\id\otimes \hat{f}_{\tau,\theta}\otimes \hat{f}_{\tau,\theta})\circ(\Delta^{\tee{\b}{\p}}_{R,S}\otimes \id)\circ\Delta^{\tee{\b}{\p}}_{RS,T} \\
&= (\id\otimes \mu^{\tee{\d}{\q}}_{S,T})\circ(\id\otimes \hat{f}_{\tau,\theta}\otimes \hat{f}_{\tau,\theta})\circ(\id \otimes \Delta^{\tee{\b}{\p}}_{S,T})\circ\Delta^{\tee{\b}{\p}}_{R,ST} \\
&= (\id\otimes \mu^{\tee{\d}{\q}}_{S,T})\circ(\id\otimes \Delta^{\tee{\d}{\q}}_{S,T})\circ(\id \otimes \hat{f}_{\tau,\theta})\circ\Delta^{\tee{\b}{\p}}_{R,ST} \\
&= (\id\otimes \hat{f}_{\tau,\theta})\circ\Delta^{\tee{\b}{\p}}_{R,ST} 
\\
&=\port_s.
\end{align*}
The penultimate step needs some explanation. While \cite[Cor.~8.38(i)]{aguiar2010monoidal} provides that $\Delta^{\h}_{A,B}\mu^{\h}_{A,B}=\id$ for any bimonoid $\h$,\footnote{We thank Mark Denker for sharing this identity with us.} these mappings appear in the reverse order for us. Focusing on the $\b,\d$ components of $\mu^{\tee{\d}{\q}}_{S,T} \circ \Delta^{\tee{\d}{\q}}_{S,T} \circ \hat{f}_{\tau,\theta}$, we analyze the map
\[
    \mu^{\d}_{X'',X'''}\circ\Delta^{\d}_{X'',X'''}\circ\tau\circ\mu^{\b}_{X^1,\ldots, X^k}\circ\Delta^b_{X^1,\ldots, X^k}.
\]
Here $(X^1\sqcup \cdots \sqcup X^{r-1})\vdash T$ and $(X^r\sqcup \cdots \sqcup X^{k})\vdash S$, with $k=s-1$.) Since $\tau$ is a bimonoid map, we have
\begin{align*}
\mu^{\d}_{X'',X'''}\circ\Delta^{\d}_{X'',X'''}\circ\tau\circ\mu^{\b}_{X^1,\ldots, X^k}
&=
\mu^{\d}_{X'',X'''}\circ\Delta^{\d}_{X'',X'''}\circ\mu^{\d}_{X^1,\ldots, X^k}\circ\tau.
\end{align*}
Also, since $\d$ is commutative, we have $\mu^{\d}_{X^1,\ldots, X^k} = \mu^{\d}_{X'',X'''}\circ(\mu^{\d}_{X^r,\ldots, X^{k}}\otimes \mu^{\d}_{X^1,\ldots, X^{r-1}})$. So
the middle terms $\Delta^{\d}_{X'',X'''}\circ\mu^{\d}_{X'',X'''}$ cancel, and we're left with what we need.
\end{proof}

\begin{remark} If $\d = \bfE$, then $\hat{f}_{\tau,\theta} = f_{\tau,\theta}$ (and is surjective whenever $\theta$ is), since $\cfS(\q)$ is a free commutative monoid over $\q$.
\end{remark}

\begin{example}
    We close this section by interpolating between our running examples $\tee{\bfG}{\bfL_+}$ and $\tee{\bfE}{\cyc}$. Let $\theta: \bfL_+\to\cyc$ be the map sending a linear order to the obvious cycle. The reader may check that $\theta$ is a map of positive comonoids. Let $\tau:\bfG\to \bfE$ be the bimonoid map sending each $g\in\rfG[I]$ to $\ast_I \in \rfE[I]$. The reader may check that $\tau$ also intertwines with restrictions. Thus, $\bigl\{\rspec{\tee{\bfG,\bfE}{\bfL_+{,}\cyc}}\bigr\}_{r\geq1}$ is a one-parameter family of Hopf monoids interpolating between $\tee{\bfG}{\bfL_+}$ and $\tee{\bfE}{\cyc}$. 
    See Figure \ref{fig:GoL-interpolation}.
\begin{figure}[htb]
\tikzset{every picture/.style={scale=.75}}
\begin{subcaptionblock}[b]{.25\textwidth}{%
\centering
\begin{tikzpicture}[vert/.style={draw,fill,circle,inner sep=0,minimum height=5pt},
    	purp/.style={color=purple!30!white,inner sep=.75pt}
    	]
    
    
      \draw[style=purp,rounded corners=10,line width=.75pt]
         (0,0) rectangle (4.75,4);
      \draw[style=purp,line width=1.25pt] (4.5,0) -- (4.5,4);
    
      \node[style=vert, label=above:{\small$a$}] (a) at (.4,3.2) {};
      \node[style=vert, label=above:{\small$c|b$}] (cb) at (1.2,2.5) {};
      \node[style=vert, label=above right:{\small$e|d$}] (ed) at (2.4,3.2) {};
      \node[style=vert, label=below:{\small$y$}] (y) at (.55,1.5) {};
      \node[style=vert, label=below left:{\small$h|i$}] (hi) at (1.7,.75) {};
      \node[style=vert, label=above right:{\small$j|f|k$}] (jfk) at (2.9,1.7) {};
      \node[style=vert, label={[label distance=-0.05cm]190:{\small$m|x$}}] (xm) at (4.0,.7) {};
     
      \draw[very thick] (y) -- (a) -- (cb);
      \draw[very thick] (xm) -- (jfk);
      \draw[very thick] (hi) -- (ed) -- (jfk) -- (hi);
    \end{tikzpicture}
\caption{$r=1$.}
}\end{subcaptionblock}
\ \raisebox{13ex}{$\leadsto$} \ 
\begin{subcaptionblock}[b]{.27\textwidth}{%
\centering
    \begin{tikzpicture}[vert/.style={draw,fill,circle,inner sep=0,minimum height=5pt},
    	purp/.style={color=purple!30!white,inner sep=.75pt}
    	]
    
    
      \draw[style=purp,rounded corners=10,line width=.75pt]
         (.6,0) rectangle (5.55,4);
      \draw[style=purp,line width=1.25pt] (4.65,0) -- (4.65,4);
    
      \node[style=vert, label=above:{\small$c|b$}] (cb) at (1.2,2.5) {};
      \node[style=vert, label=above right:{\small$e|d$}] (ed) at (2.4,3.2) {};
      \node[style=vert, label=below left:{\small$h|i$}] (hi) at (1.7,.75) {};
      \node[style=vert, label=above right:{\small$j|f|k$}] (jfk) at (2.9,1.7) {};
      \node[style=vert, label={[label distance=-0.05cm]190:{\small$m|x$}}] (xm) at (4.0,.7) {};
      \draw[very thick] (xm) -- (jfk);
      \draw[very thick] (hi) -- (ed) -- (jfk) -- (hi);
      \node[label=above:{\small$(a)$}] (a) at (5.1,2.8) {};
      \node[label=below:{\small$(y)$}] (y) at (5.1,1.5) {};

    \end{tikzpicture}
\caption{$r=2$.}
}\end{subcaptionblock}
\ \raisebox{13ex}{$\leadsto$} \ 
\begin{subcaptionblock}[b]{.25\textwidth}{%
\centering
    \begin{tikzpicture}[vert/.style={draw,fill,circle,inner sep=0,minimum height=5pt},
    	purp/.style={color=purple!30!white,inner sep=.75pt}
    	]
    
    
      \draw[style=purp,rounded corners=10,line width=.75pt]
         (2.5,0) rectangle (7,4);
      \draw[style=purp,line width=1.25pt] (3.85,0) -- (3.85,4);
    
      \node[style=vert, label=above:{\small$j|f|k$}] (jfk) at (3.2,1.7) {};
      \node[label={\small$(cb)$}] (cb) at (5.1,1.9) {};
      \node[label={\small$(ed)$}] (ed) at (6,3) {};
      \node[label={\small$(hi)$}] (hi) at (5,.05) {};
      \node[label={\small$(mx)$}] (xm) at (6.3,.05) {};
      \node[label={\small$(a)$}] (a) at (4.3,2.8) {};
      \node[label={\small$(y)$}] (y) at (4.3,1.1) {};

    \end{tikzpicture}
\caption{$r=3$.}
}\end{subcaptionblock}
\caption{Interpolation in $\rspec{\tee{\bfG,\bfE}{\bfL_+{,}\cyc}}$.}
\label{fig:GoL-interpolation}
\end{figure}
\end{example}

\section{Further Questions}
\label{sec:further-questions}


\subsubsection*{C(o)operadic considerations} Recall from \cite[App. B]{aguiar2010monoidal} that $\bfL_+$ is an operad. That is, a monoid in the category $(\Spec_+,\scirc)$. So the map $\chi :\bfL_+\scirc\h_+ \to \h_+$ required in Section \ref{sec:towards-a-up} amounts to $\h_+$ being a (left) $\bfL_+$-module in $(\Spec_+,\scirc)$. Towards a universal property for $\tee{\b}{\p}$,\footnote{In \cite{aguiar2010monoidal}, Aguiar and Mahajan translate a result of Fresse \cite[Sec. 2.1.10]{fresse2004koszul} to the present context to give a universal property for our species $\b\scirc\p$ in terms of free modules over operads. But the hypotheses (on maps $\zeta$) and conclusions (on $\tee{\b}{\p}$ and maps $\hat\zeta$) are not those asked for in Sections \ref{sec:two-functors} and \ref{sec:towards-a-up}.} we ask when are linearized bimodules $\b_+$ also operads---and is it sufficient to take $\b_+$-modules as the (comonoidal) targets of $\p$?

Along the same lines, recall the Hopf monoid of planar forests $\vec{\mathbf{F}}$, which is $\cfT(\vec{\a})$ as a monoid, but has a different comonoid structure (here $\vec{\a}$ indicates planar rooted trees). Can we further generalize our $\tee{\b}{\p}$, incorporating (co)operadic structures, so that $\vec{\mathbf{F}}$ is captured as the special case $(\b=\bfL, \p=\vec{\a})$?

\subsubsection*{Relaxing hypotheses}
 The proof of Theorem \ref{th:bop} suggests the stringent hypotheses on $\b$ are necessary. However, if $\p$ is the trivial comonoid, we may allow arbitrary Hopf structures for $\b$. (Then $\tee{\b}{\p}$ is simply a $\p$-decorated  version of $\b$.) We close by asking, are there twisted versions of \eqref{eq:bop-product} and \eqref{eq:bop-coproduct} that allow one to relax the cocommutative or linearized hypothesis on $\b$?
 


\printbibliography

@book {aguiar2010monoidal,
    AUTHOR = {Aguiar, Marcelo and Mahajan, Swapneel},
     TITLE = {Monoidal functors, species and {H}opf algebras},
    SERIES = {CRM Monograph Series},
    VOLUME = {29},
      NOTE = {With forewords by Kenneth Brown and Stephen Chase and Andr\'{e}
              Joyal},
 PUBLISHER = {American Mathematical Society, Providence, RI},
      YEAR = {2010},
     PAGES = {lii+784},
      ISBN = {978-0-8218-4776-3},
       DOI = {10.1090/crmm/029},
}

@book {bergeron1998combinatorial,
    AUTHOR = {Bergeron, F. and Labelle, G. and Leroux, P.},
     TITLE = {Combinatorial species and tree-like structures},
    SERIES = {Encyclopedia of Mathematics and its Applications},
    VOLUME = {67},
      NOTE = {Translated from the 1994 French original by Margaret Readdy,
              With a foreword by Gian-Carlo Rota},
 PUBLISHER = {Cambridge University Press, Cambridge},
      YEAR = {1998},
     PAGES = {xx+457},
      ISBN = {0-521-57323-8},
}

@incollection {fresse2004koszul,
    AUTHOR = {Fresse, Benoit},
     TITLE = {Koszul duality of operads and homology of partition posets},
 BOOKTITLE = {Homotopy theory: relations with algebraic geometry, group
              cohomology, and algebraic {$K$}-theory},
    SERIES = {Contemp. Math.},
    VOLUME = {346},
     PAGES = {115--215},
 PUBLISHER = {Amer. Math. Soc., Providence, RI},
      YEAR = {2004},
      ISBN = {0-8218-3285-9},
       DOI = {10.1090/conm/346/06287},
}

@article {hivert2008commutative,
    AUTHOR = {Hivert, Florent and Novelli, Jean-Christophe and Thibon,
              Jean-Yves},
     TITLE = {Commutative combinatorial {H}opf algebras},
   JOURNAL = {J. Algebraic Combin.},
  FJOURNAL = {Journal of Algebraic Combinatorics. An International Journal},
    VOLUME = {28},
      YEAR = {2008},
    NUMBER = {1},
     PAGES = {65--95},
      ISSN = {0925-9899},
       DOI = {10.1007/s10801-007-0077-0},
}

@article{joni1979coalgebras,
    AUTHOR = {Joni, Saj-nicole A. and Rota, Gian-Carlo},
     TITLE = {Coalgebras and bialgebras in combinatorics},
   JOURNAL = {Stud. Appl. Math.},
  FJOURNAL = {Studies in Applied Mathematics},
    VOLUME = {61},
      YEAR = {1979},
    NUMBER = {2},
     PAGES = {93--139},
      ISSN = {0022-2526},
       DOI = {10.1002/sapm197961293},
}

@article{joyal1981species,
    AUTHOR = {Joyal, Andr\'{e}},
     TITLE = {Une th\'{e}orie combinatoire des s\'{e}ries formelles},
   JOURNAL = {Adv. in Math.},
  FJOURNAL = {Advances in Mathematics},
    VOLUME = {42},
      YEAR = {1981},
    NUMBER = {1},
     PAGES = {1--82},
      ISSN = {0001-8708},
       DOI = {10.1016/0001-8708(81)90052-9},
}

@unpublished{lauve202xinterpolation,
  author = {Lauve, Aaron and Lazzeroni, Anthony},
  title  = {Interpolation in Species and a lift of the {H}opf algebra of $r$-Quasisymmetric Functions},
  note   = {{\it In preparation}},
  month  = {},
  year   = {},
}

@article {marberg2016duality,
    AUTHOR = {Marberg, Eric},
     TITLE = {Strong forms of self-duality for {H}opf monoids in species},
   JOURNAL = {Trans. Amer. Math. Soc.},
  FJOURNAL = {Transactions of the American Mathematical Society},
    VOLUME = {368},
      YEAR = {2016},
    NUMBER = {8},
     PAGES = {5433--5473},
      ISSN = {0002-9947},
       DOI = {10.1090/tran/6506},
}

@article {marberg2015linearization,
    AUTHOR = {Marberg, Eric},
     TITLE = {Strong forms of linearization for {H}opf monoids in species},
   JOURNAL = {J. Algebraic Combin.},
  FJOURNAL = {Journal of Algebraic Combinatorics. An International Journal},
    VOLUME = {42},
      YEAR = {2015},
    NUMBER = {2},
     PAGES = {391--428},
      ISSN = {0925-9899},
       DOI = {10.1007/s10801-015-0585-2},
}

@article{schmitt1993hopf,
    AUTHOR = {Schmitt, William R.},
     TITLE = {Hopf algebras of combinatorial structures},
   JOURNAL = {Canad. J. Math.},
  FJOURNAL = {Canadian Journal of Mathematics. Journal Canadien de
              Math\'ematiques},
    VOLUME = {45},
      YEAR = {1993},
    NUMBER = {2},
     PAGES = {412--428},
      ISSN = {0008-414X,1496-4279},
       DOI = {10.4153/CJM-1993-021-5},
}

@article{stover1993equivalence,
    AUTHOR = {Stover, Christopher R.},
     TITLE = {The equivalence of certain categories of twisted {L}ie and
              {H}opf algebras over a commutative ring},
   JOURNAL = {J. Pure Appl. Algebra},
  FJOURNAL = {Journal of Pure and Applied Algebra},
    VOLUME = {86},
      YEAR = {1993},
    NUMBER = {3},
     PAGES = {289--326},
      ISSN = {0022-4049,1873-1376},
       DOI = {10.1016/0022-4049(93)90106-4},
}

@article {white2020cohen,
    AUTHOR = {White, Jacob A.},
     TITLE = {On {C}ohen-{M}acaulay {H}opf monoids in species},
   JOURNAL = {S\'{e}m. Lothar. Combin.},
  FJOURNAL = {S\'{e}minaire Lotharingien de Combinatoire},
    VOLUME = {84B},
      YEAR = {2020},
     PAGES = {Art. 84, 12},
}

@article {white2025chromatic,
    AUTHOR = {White, Jacob A.},
     TITLE = {Chromatic quasisymmetric class functions for combinatorial
              {H}opf monoids},
   JOURNAL = {European J. Combin.},
  FJOURNAL = {European Journal of Combinatorics},
    VOLUME = {124},
      YEAR = {2025},
     PAGES = {Paper No. 104055, 23},
      ISSN = {0195-6698},
       DOI = {10.1016/j.ejc.2024.104055},
}

\end{document}